\documentclass[reqno]{amsart}
\usepackage{amsmath,amssymb,amsthm}
\usepackage{thmtools}
\usepackage{mathptmx}      
\usepackage{mathrsfs}
\usepackage[scr=boondox]{mathalpha}
\usepackage{latexsym}
\usepackage{amssymb}
\usepackage{bm}
\usepackage[usenames,dvipsnames]{xcolor}
\usepackage{graphicx,pifont}
\usepackage{multicol,multirow}
\usepackage{ulem,cancel}
\usepackage{tikz}
\usetikzlibrary[patterns]
\usepackage{comment}
\usepackage{array}
\usepackage{subcaption}
\usepackage{txfonts}
\usepackage{mathtools}
\usepackage[shortlabels]{enumitem}
\usepackage{hyperref}
\usepackage{booktabs}

\usepackage{cleveref}
\crefname{equation}{}{}
\crefname{enumi}{}{}
\crefname{figure}{Figure}{Figure}
\crefname{table}{Table}{Table}
\crefname{subsection}{Subsection}{Subsections}
\crefname{lemma}{Lemma}{Lemma}
\crefname{theorem}{Theorem}{Theorem}
\crefname{proposition}{Proposition}{Proposition}
\crefname{section}{Section}{Section}
\crefname{appendix}{Appendix}{Appendix}
\crefname{remark}{Remark}{Remark}

\newtheorem{theorem}{Theorem}[section]
\newtheorem{lemma}[theorem]{Lemma}

\theoremstyle{definition}

\theoremstyle{remark}

\numberwithin{equation}{section}
\numberwithin{figure}{section}
\numberwithin{table}{section}

\newcommand{\bs}[1]{\boldsymbol{#1}}

\newcommand{\oname}[1]{\mathrm{#1}}
\newcommand{\abs}[1]{\left|#1\right|}
\newcommand{\nrm}[1]{\left|\left|#1\right|\right|}
\newcommand{\eqdef}{\stackrel{\mathrm{def}}{=\joinrel=}}

\newcommand{\phf}[1]{{{#1}+1/2}}

\newcommand{\floor}[1]{\left\lfloor{#1}\right\rfloor}

\newcommand{\cfl}{{\textrm{cfl}}}

\newcommand{\cm}{\textrm{\ding{51}}}
\newcommand{\xm}{\textrm{\ding{55}}}
\newcommand{\hc}{\mathrlap{\cm}{~{\large\backprime}}}

\begin{document}

\title[Stability Barrier of Hermite Type Methods]{The Stability Barrier of A Class of Hermite-Type Discretizations of Advection Equations}

\author[X.~Zeng]{Xianyi Zeng}
\address{Department of Mathematics, Lehigh University, Bethlehem, PA 18015, United States.}
\email[Corresponding author, X.~Zeng]{xyzeng@lehigh.edu}

\date{\today}



\begin{abstract}
  In this paper we fully categorize all hybrid-variable (HV) discretizations of linear advection equations according to their stability property and establish the stability barrier of these methods.
  In the HV discretization framework, we find numerical approximations to both cell-averaged solutions and nodal solutions and evolve them in time simultaneously.
  We prove that the HV methods have a similar stability barrier as the classical finite difference schemes, that the semi-discretized HV scheme is stable if and only if it uses a stencil such that the number of unknowns in the upwind direction is precisely that in the downwind direction plus one or two, with only one exception when the two numbers are three and zero, respectively.
  It is also proved that all central HV schemes are neutrally stable, in the sense that the $L^2$-norm of the solution remain uniformly bounded at all times.
  The theoretical predictions are verify by extensive numerical tests.
\end{abstract}

\maketitle

\section{Introduction}
\label{sec:intro}
In this work we establish the precise stability barrier of the recently developed hybrid-variable (HV) discretizations of linear advection equations, a prototype problem for general hyperbolic conservation laws that has wide applications in fluid mechanics, celluar dynamics in tumor growth models, and population dynamics.

Unlike traditional finite-difference or finite-volume schemes, the hybrid-variable methods approximate both the cell-averaged solutions and nodal solutions and evole them simultaneously in time in the context of method of lines.
To this end, the HV schemes are direct generalizations of the Scheme V of van Leer~\cite{BvanLeer:1977b}, which uses both cell averages and nodal values to construct piecewise quadratic representations of the solution.
The construction of HV discretization of arbitrary order utilizes extensively the Hermite interpolation theory and is discussed in detail in~\cite{XZeng:2019a}, which also proved the supraconvergence property of these methods.
To be more precise, the construction of HV schemes builds on approximating the spatial derivative at nodes by a linear combination of nearby cell-averaged and nodal solutions.
Let the stencil of the method be denoted $(L,R)$, meaning $L$ unknowns (combining both types of variables) are used in the upwind direction and $R$ unknowns are used in the downwind direction, previous work proved that the resulting HV scheme has spatial order of accuracy $L+R+1$ provided that it is upwind-biased ($L>R$) and stable.
Note that this spatial order is one higher than the local truncation error ($L+R$), a property known as supraconvergence in the literature~\cite{HOKreiss:1986a}.

Although~\cite{XZeng:2019a} proved supraconvergence for all stable HV schemes, it did not tackle the general stability theory of these methods.
Later on, the author proved using Fourier analysis and combinatoric tools that all HV schemes with $L\gg R$ are unstable~\cite{XZeng:2026a}, indicating the existence of a stability barrier of HV methods, i.e., the number of upwind unknowns cannot be too much larger than that of downwind unknowns for the method to be stable.
The stability barrier has been studied for the classical finite difference schemes, and fully resolved by Iserles and Strang~\cite{AIserles:1982a,AIserles:1983a} using the theory of order stars and also by Despr\'{e}s~\cite{BDespres:2008a,BDespres:2009a} using integral form of the local truncation error.
Neither method, however, extends to the HV methods, as discussed in~\cite[Appendix A]{XZeng:2026a}.

In this work, we employs extensively techniques in the studying of hypergeometric series and their relation to differential equations, and establish the precise stability barrier for HV schemes, and prove that: (1) all HV schemes with $1\le L-R\le 2$ are stable, (2) all HV schemes with $L-R\ge3$ are unstable except when $L=3$ and $R=0$, and (3) all central HV schemes (i.e., $L=R$) are neutrally stable in the sense that the $L^2$-norm of the numerical solution is uniformly bounded at all times.
More precisely, we review the construction of HV schemes in~\cref{sec:prelim}, which also outlines the core techniques used in later proofs.
Particularly, the instability proof when $L-R>2$ is provided next in~\cref{sec:instab}, and the stability proof when $1\le L-R\le2$ is offered in~\cref{sec:stab}.
Next, we prove that all central HV schemes are neutrally stable, in the sense that the $L^2$-norm of the numerical solution is uniformly bounded for all $t>0$; this result is also extended to fully-discretized central HV schemes.
Extensive numerical examples are offered in~\cref{sec:num} to verify the theoretical results proved in previous sections.
And lastly,~\cref{sec:concl} concludes the paper.


\section{Preliminaries}
\label{sec:prelim}
Let us consider the Cauchy problem for the advection equation:
\begin{equation}\label{eq:prelim_adv}
  u_t + u_x = 0\;, 
\end{equation}
on the domain $[0,\,1]$ with periodic boundary condition $u(0,t)=u(1,t)$; and we seek numerical approximations with respect to a uniform grid $x_j=jh$, where $h=1/N>0$ is the grid size.
The discretization of~\cref{eq:prelim_adv} in space by the hybrid-variable (HV) method concerns approximations to both the nodal values $u_j(t)\approx u_j^\ast(t) \eqdef u(x_j,t)$ and cell averages $\overline{u}_{\phf{j}}(t)\approx \overline{u}^\ast_{\phf{j}} \eqdef \frac{1}{h}\int_{x_j}^{x_{j+1}}u(x,t)dx$.
It is understood that $u_j=u_{j+N}$ and $\overline{u}_{\phf{j}}=\overline{u}_{\phf{j+N}}$ for all integer $j$.
Here the asteroid denotes exact values.
An HV scheme is determined by the discrete differential operator (DDO) $[\mathcal{D}_x]$, which is defined by (suppressing the dependence on $t$):
\begin{equation}\label{eq:prelim_ddo}
  [\mathcal{D}_xu]_j \eqdef \frac{1}{h}\sum_{k=-l}^{r-1}\alpha_k\overline{u}_{\phf{j+k}}+\frac{1}{h}\sum_{k=-l'}^{r'}\beta_ku_{j+k}\;.
\end{equation}
The construction of $[\mathcal{D}_xu]_j\approx u_x(x_j)$ uses $l$ cells and $l'$ nodes to the left of $x_j$, and $r$ cells and $r'$ nodes to the right of $x_j$; the quadruple $(l,r,l',r')$ is thus called the stencil of the DDO $[\mathcal{D}_x]$.
We assume a continuous stencil, i.e., $l'\le l\le l'+1$ and $r'\le r\le r'+1$.
An equivalent way to denote such $[\mathcal{D}_x]$ is the pair of indices $(L,R)$, where $L=l+l'$ and $R=r+r'$; one can determine the four-integer index by $l'=\floor{L/2}$, $l=L-l'$, $r'=\floor{R/2}$, and $r=R-r'$.

The DDO~\cref{eq:prelim_ddo} is $p$-th order accurate if for sufficiently smooth $u$, one has:
\begin{displaymath}
  [\mathcal{D}_xu^\ast]_j = u_x(x_j) + c_p\frac{d^{p+1}u(x_j)}{dx^{p+1}}h^p + O(h^{p+1})
\end{displaymath}
for some $c_p\ne0$.
In previous work~\cite{XZeng:2019a} it is shown that given a continuous stencil $(l,r,l',r')$ there exists a unique $[\mathcal{D}_x]$ that has the optimal order $p=l+r+l'+r'=L+R$.
We shall only consider these optimally accurate DDO in the present work.
Once $[\mathcal{D}_x]$ is chosen, the semi-discretization of~\cref{eq:prelim_adv} is thus given by for all $1\le j\le N$:
\begin{equation}\label{eq:prelim_semi}
  \left\{\begin{array}{l}
    \overline{u}_{\phf{j}}' + \frac{1}{h}\left(u_{j+1}-u_j\right) = 0\;, \\ \vspace*{-.05in} \\
    u_j' + [\mathcal{D}_xu]_j = 0\;.
  \end{array}\right.
\end{equation}
Using Fourier analysis, much of the property of the numerical method can be obtained by studying the eigenvalues of this linear system of ordinary differential equations.
Particularly, one can show that (for example, following~\cite[Section 4]{XZeng:2019a}) all eigenvalues of the ODE system~\cref{eq:prelim_semi} lie on the set of eigenvalues (denoted $\mathcal{S}$) of the $2\times2$ matrix:
\begin{equation}\label{eq:prelim_semi_mat}
  \begin{bmatrix}
    0 & e^{i\theta}-1 \\
    G(e^{i\theta}) & H(e^{i\theta})
  \end{bmatrix}\;,\quad\textrm{ where }
  G(z) = \sum_{k=-l}^{r-1}\alpha_kz^k\;,\ \textrm{ and }\  
  H(z) = \sum_{k=-l'}^{r'}\beta_kz^k\;,
\end{equation}
with $\theta$ runs from $0$ to $2\pi$. 
The characteristic polynomial of this matrix is given by:
\begin{equation}\label{eq:prelim_semi_char}
  \lambda^2 - H(e^{i\theta})\lambda - (e^{i\theta}-1)G(e^{i\theta}) = 0\;,
\end{equation}
and the set $\mathcal{S}$ is thus defined as:
\begin{equation}\label{eq:prelim_semi_eigset}
  \mathcal{S} = \{\lambda\in\mathbb{C}:\, \lambda^2-H(e^{i\theta})\lambda-F(e^{i\theta}) = 0\ \textrm{ for some }\ 0\le\theta\le2\pi\},
\end{equation}
where $F(e^{i\theta})=(e^{i\theta}-1)G(e^{i\theta})$.
Physically, $\theta=\kappa h$ where $\kappa$ represents the wave number.

It was proved that given sufficiently smooth initial data, the semi-discretized solution converges to the actual one pointwise at $(p+1)$-th order of accuracy if: (1) the DDO $[\mathcal{D}_x]$ is $p$-th order, (2) the stencil is biased towards the upwind direction, or $L>R$, and (3) the ODE system is stable (to be made precise later).
Here the first condition is known as {\it supraconvergence}, as the order of the method is higher than that of the local truncation error.
The upwind-biased condition, despite its similarity to that of finite-difference method, is not due to stability considerations but a necessary condition for the second root of~\cref{eq:prelim_semi_char} to decay exponentially fast to zero as $h\to0$ or $\theta\to0$.
This paper focuses on the third condition and investigates the stability given an upwind biased stencil ($L>R$).

When saying the linear ODE system~\cref{eq:prelim_semi} is stable, we mean all its eigenvalues have negative real parts except at most a simple one at zero.
Note that these eigenvalues are precisely $-\lambda(\theta)$, where $\lambda$ is a root of~\cref{eq:prelim_semi_char}; and it is easy to compute when $\theta=0$, the two roots of~\cref{eq:prelim_semi_char} are $0$ and $H(1)=\sum_{k=-l'}^{r'}\beta_k>0$\footnote{The two conditions $H(1)>0$ and $L>R$ are equivalent, see~\cite[Proposition 5]{XZeng:2019a}.}, it is not difficult to see that stability for all admissible combinations of wavenumber, domain size, and number of cells is equivalent to say $\mathcal{S}'$ is contained in the open right complex plane $\mathbb{C}^+=\{z:\,\oname{Re}z>0\}$, where:
\begin{equation}\label{eq:prelim_semi_traj}
  \mathcal{S}' = \{\lambda\in\mathbb{C}:\, \lambda^2-H(e^{i\theta})\lambda-F(e^{i\theta}) = 0\ \textrm{ for some }\ 0<\theta<2\pi\}.
\end{equation}

A sufficient and necessary condition for $\mathcal{S}'\subseteq\mathbb{C}^+$ is given in previous work~\cite{XZeng:2026a}:
\begin{lemma}\label{lm:prelim_stab}
  The set $\mathcal{S}'$ is contained in $\mathbb{C}^+$ if and only if for all $0<\theta<2\pi$, the following two conditions are satisfied:
  \begin{subequations}\label{eq:prelim_stab_cond}
    \begin{align}
      \label{eq:prelim_stab_cond_h}
      \oname{Re}H > 0\;, \\
      \label{eq:prelim_stab_cond_g}
      \oname{Re}H\,\oname{Re}(\overline{H}F)+(\oname{Im}F)^2 < 0\;.
    \end{align}
  \end{subequations}
  Here $H=H(e^{i\theta})$ and $F=F(e^{i\theta})$.
\end{lemma}
Throughout the paper we will write:
\begin{equation}\label{eq:prelim_stab_parts}
  h_r(\theta) = \oname{Re}H(e^{i\theta})\;,\quad
  h_i(\theta) = \oname{Im}H(e^{i\theta})\;,\quad
  f_r(\theta) = \oname{Re}F(e^{i\theta})\;,\quad
  f_i(\theta) = \oname{Re}F(e^{i\theta})\;,
\end{equation}
then the two conditions~\cref{eq:prelim_stab_cond} become:
\begin{equation}\label{eq:prelim_stab_cond_equiv}
  h_r > 0\;,\quad q(\theta)\eqdef h_r(h_rh_i+f_rf_i)+f_i^2 < 0\;,\quad\forall 0<\theta<2\pi\;.
\end{equation}

The question about the stability barrier concerns whether a DDO $[\mathcal{D}_x]$ with optimal order of accuracy and a stencil such that $L>R$ will be stable.
This paper establishes the precise stability barrier for HV methods. 
We begin with recalling that the $\alpha$- and $\beta$-coefficients can be obtained using Hermite interpolation polynomials, see also~\cite[Theorem 1]{XZeng:2019a}:
\begin{lemma}\label{lm:prelim_coef}
  The unique DDO $[\mathcal{D}_x]$ with stencil $(l,r,l',r')$ and optimal order of accuracy $p=l+r+l'+r'$ is given by~\cref{eq:prelim_ddo} and:
  \begin{subequations}\label{eq:prelim_coef}
    \begin{align}
      \label{eq:prelim_coef_cell_neg}
      \alpha_\nu &= -(1-\delta_{ll'})\frac{2}{l^2}C^{l,r}_{-l}C^{l,r'}_{-l}-\sum_{k=-l'}^\nu\frac{2(1+k(\zeta_k^{l,r}+\zeta_k^{l',r'}))}{k^2}C_k^{l,r}C_k^{l',r'},\ \forall -l\le\nu<0\;; \\
      \label{eq:prelim_coef_cell_pos}
      \alpha_\nu &= \sum_{k=\nu+1}^{r'}\frac{2(1+k(\zeta_k^{l,r}+\zeta_k^{l',r'}))}{k^2}C_k^{l,r}C_k^{l',r'}+(1-\delta_{rr'})\frac{2}{r^2}C^{l,r}_{r}C^{l',r}_{r},\ \forall 0\le\nu<r\;; \\
      \label{eq:prelim_coef_node} 
      \beta_0 &= 2(\zeta_0^{l,r}+\zeta_0^{l',r'})\;;\quad
      \beta_\nu = -\frac{2}{\nu}C^{l,r}_{\nu}C^{l',r'}_{\nu}\;,\ \forall -l'\le\nu\le r',\ \nu\ne0\;.
    \end{align}
  \end{subequations}
  Here $\delta_{ab}$ is Kronecker delta, $C^{m,n}_k=\frac{m!n!}{(m+k)!(n-k)!}$, and $\zeta_k^{m,n} = H_{m+k}-H_{n-k}$ for all $-m\le k\le n$, where $H_k=1+\cdots+1/k$ is the $k$-th Harmonic number and it is understood that $H_0=0$.
\end{lemma}
It is convenient to rewrite the $\alpha$-coefficients as:
\begin{align}
  \label{eq:prelim_coef_alpha_sum}
  &\alpha_\nu = -\sum_{k=-l}^\nu t_k\;,\quad -l\le\nu<0\;;\qquad \alpha_\nu = \sum_{k=\nu+1}^r t_k\;,\quad 0\le\nu<r\;; \\ 
  \label{eq:prelim_coef_t_neg}
  & \left\{\begin{array}{lcl}
    \textrm{If } l'=l: & & t_k = \frac{2(1+k(\zeta_k^{l,r}+\zeta_k^{l',r'}))}{k^2}C_k^{l,r}C_k^{l',r'}\;,\quad -l\le k\le -1\;, \\
    \textrm{If } l'=l-1: & & t_k = \left\{\begin{array}{lcl}
      \frac{2(1+k(\zeta_k^{l,r}+\zeta_k^{l',r'}))}{k^2}C_k^{l,r}C_k^{l',r'}\;, & & -(l-1)\le k\le -1\;, \\
      \frac{2}{l^2}C_{-l}^{l,r}C_{-l}^{l,r'}\;, & & k=-l\;.
    \end{array}\right.
  \end{array}\right. \\
  \label{eq:prelim_coef_t_pos}
  & \left\{\begin{array}{lcl}
    \textrm{If } r'=r: & & t_k = \frac{2(1+k(\zeta_k^{l,r}+\zeta_k^{l',r'}))}{k^2}C_k^{l,r}C_k^{l',r'}\;,\quad 1\le k\le r\;, \\
    \textrm{If } r'=r-1: & & t_k = \left\{\begin{array}{lcl}
      \frac{2(1+k(\zeta_k^{l,r}+\zeta_k^{l',r'}))}{k^2}C_k^{l,r}C_k^{l',r'}\;, & & 1\le k\le r-1\;, \\
      \frac{2}{r^2}C_{r}^{l,r}C_{r}^{l',r}\;, & & k=r\;.
    \end{array}\right.
  \end{array}\right. 
\end{align}
Using this notation, one can write:
\begin{equation}\label{eq:prelim_f_t}
  F(e^{i\theta}) = \sum_{k=-l,k\ne0}^rt_k(e^{ik\theta}-1)\;.
\end{equation}

The stability of HV methods with stencil $(L,R)$, $0\le R<L\le8$ are checked and listed in~\cref{tb:prelim_stab}.
\begin{table}\centering
  \caption{HV methods: $\cm$ stable, $\hc$ unstable but satisfies~\cref{eq:prelim_stab_cond_h}, $\xm$ unstable and violates~\cref{eq:prelim_stab_cond_h}.}
  \label{tb:prelim_stab}
  \begin{tabular}{@{}lllllllll@{}}
    \toprule[.5mm]
        & \multicolumn{7}{c}{$R$} \\ \cmidrule[.2mm]{2-9}
    $L$ & 0     & 1     & 2     & 3     & 4     & 5     & 6     & 7     \\ \cmidrule[.3mm]{1-9}
    1   & $\cm$ &       \\
    2   & $\cm$ & $\cm$ &     \\
    3   & $\cm$ & $\cm$ & $\cm$ &       \\
    4   & $\xm$ & $\hc$ & $\cm$ & $\cm$ &       \\
    5   & $\xm$ & $\xm$ & $\hc$ & $\cm$ & $\cm$ &       \\
    6   & $\xm$ & $\xm$ & $\xm$ & $\hc$ & $\cm$ & $\cm$ &       \\
    7   & $\xm$ & $\xm$ & $\xm$ & $\xm$ & $\hc$ & $\cm$ & $\cm$ &       \\
    8   & $\xm$ & $\xm$ & $\xm$ & $\xm$ & $\xm$ & $\hc$ & $\cm$ & $\cm$ \\
    \bottomrule[.4mm]
  \end{tabular}
\end{table}
The main result of this paper is given below:
\begin{theorem}\label{thm:prelim_sb}
  The semi-discretized HV method is stable if and only if $0\le R < L < R+3$ except $(L,R)=(3,0)$, when the method is stable. 
\end{theorem}
This result is very similar to that of the finite-difference method in the literature: an optimally accurate finite-difference method is stable if and only if $R<L<R+3$.
It was first proved using the theory of order stars~\cite{AIserles:1982a,AIserles:1983a} and then by a more elementary method analyzing the integral form of an error function~\cite{BDespres:2009a}; however, neither strategy seems to work in the case of HV methods, as discussed in the previous work.
We adopt a different approach in this paper.

It should be noted that finite-difference methods with central stencil are sometimes considered as stable methods as they preserve the square sum of nodal solutions.
We will prove at the end of the paper a similar result for central HV methods (i.e., $L=R$).

Lastly, the stability and instability proofs in the rest of paper involves a variety of techniques, which are summarized in~\cref{tb:prelim_org}.
\begin{table}\centering
  \caption{Roadmap to stability/instability proofs.
  Note the second row excludes $(L,R)=(3,0)$.}
  \label{tb:prelim_org}
  \begin{tabular}{@{}lcccc@{}}
    \toprule[.5mm]
    & Int. form of $h_r$. & Int. form of $g/h_r^2$. & $g$ when $\theta\approx0$. & Spectral analysis. \\ \cmidrule[.2mm]{2-5}
    $L\!-\!R\!>\!3$: \cref{eq:prelim_stab_cond_h}$\xm$ & \cref{sec:instab_h} \\
    $L\!-\!R\!=\!3$: \cref{eq:prelim_stab_cond_g}$\xm$ & & & \cref{sec:instab_f} \\
    $L\!-\!R\!=\!1,2$: \cref{eq:prelim_stab_cond_h}$\cm$ & \cref{sec:stab_h} \\
    $L\!-\!R\!=\!1,2$: \cref{eq:prelim_stab_cond_g}$\cm$ & & \cref{sec:stab_f} \\ 
    $L-R=0$. & & & & \cref{sec:l2} \\
    \bottomrule[.4mm]
  \end{tabular}
\end{table}
As we can see from the table, most proofs are standalone except for the two involving $h_r(\theta)$. 
The integral representation of $h_r$ will be given first in~\cref{sec:instab_int} for general stencils, and it is used in~\cref{sec:instab_h} and~\cref{sec:stab_h} given the stencils $L-R=3$ and $L-R=1,2$, respectively.

\section{Instability proof when $L-R>2$}
\label{sec:instab}
We prove the instabilities first as the derivations are cleaner.
We begin with an integral representation of $h_r(\theta)=\oname{Re}H(e^{i\theta})$ for general stencils, then use it to show $h_r(\theta^\ast)<0$ when $L-R>3$ for a carefully chosen $\theta^\ast$ motivated by the previous work~\cite{XZeng:2026a}, and lastly prove $q(\theta)>0$ for small $\theta>0$ when $L-R=3$ and $R\ne0$.

\subsection{An integral representation of $h_r$}
\label{sec:instab_int}
We compute an integral form of $h_r(\theta)$ by exploring its connection with the Beta functions.
Let $(l,r,l',r')$ be fixed such that $L\ge R$, then:
\begin{align}
  \notag
  h_r(\theta) &= \sum_{k=-l'}^{r'}\beta_k\cos{k\theta} = \beta_0 + \sum_{k=1}^{l'}\beta_{-k}\cos{k\theta} + \sum_{k=1}^{r'}\beta_k\cos{k\theta} \\
  \label{eq:instab_int_re}
  &= 2(\zeta_0^{l,r}+\zeta_0^{l',r'})+\sum_{k=1}^{l'}\frac{2}{k}\left(C^{l,r}_{-k}C^{l',r'}_{-k}-C^{l,r}_kC^{l',r'}_k\right)\cos{k\theta}\;,
\end{align}
here we adopted the convention:
\begin{equation}\label{eq:instab_int_c_conv}
  C^{m,n}_k = 0\;,\quad \textrm{ if } k < -m \textrm{ or } k > n\;.
\end{equation}

In preparation, we define for all $m,n\ge0$, 
\begin{equation}\label{eq:instab_int_c}
  \mathscr{C}^{m,n}(x) = \frac{\Gamma(m+1)\Gamma(n+1)}{\Gamma(m+x+1)\Gamma(n-x+1)}\;,
\end{equation}
where $\Gamma$ is the gamma function.
Then it is easy to verify that (1) $\mathscr{C}^{m,n}(x)$ is entire, i.e., analytic at all $x\in\mathbb{C}$, (2) for all $k\in\mathbb{Z}$, $\mathscr{C}(k) = C^{m,n}_k$ (note that if $k>n$ or $k<-m$, it is a pole of the denominator and thus a zero of $\mathscr{C}^{m,n}$). 

Denote by $\mathcal{L}=\frac{d}{dx}\ln$ the logarithmic derivative, one can compute:
\begin{displaymath}
  \mathcal{L}\mathscr{C}^{m,n}(x) = -\mathcal{L}\Gamma(m+x+1)-\mathcal{L}\Gamma(n-x+1) = -\psi(m+x+1)+\psi(n-x+1)\;,
\end{displaymath}
where $\psi$ is the digamma function.
Therefore for all $-m\le k\le n$:
\begin{equation}\label{eq:instab_int_dc}
  \mathcal{L}\mathscr{C}^{m,n}(k) = H_{n-k}-H_{m+k} = -\zeta^{m,n}_k\;.
\end{equation}

With these preparations, we define the entire function:
\begin{equation}\label{eq:instab_int_a}
  \mathscr{A}(x) = \mathscr{C}^{l,r}(x)\mathscr{C}^{l',r'}(x)\;,
\end{equation}
then there is:
\begin{align}
  \label{eq:instab_int_a_val}
  &\mathscr{A}(k) = C^{l,r}_kC^{l',r'}_k\,,\quad \forall k\in\mathbb{Z}\,; \\
  \label{eq:instab_int_a_der}
  &\mathcal{L}\mathscr{A}(k) = \mathcal{L}\mathscr{C}^{l,r}(k) + \mathcal{L}\mathscr{C}^{l',r'}(k) = - \zeta^{l,r}_k-\zeta^{l',r'}_k\,,\quad \forall -l'\le k\le r'\;,
\end{align}
and therefore:
\begin{equation}\label{eq:instab_int_hr_a}
  h_r(\theta) = -2\mathcal{L}\mathscr{A}(0) + \sum_{k=1}^{l'}\frac{2\left[\mathscr{A}(-k)-\mathscr{A}(k)\right]}{k}\cos{k\theta}\;.
\end{equation}
To this end, we define
\begin{equation}\label{eq:instab_int_b}
  \mathscr{B}(x) = \left\{\begin{array}{lcl}
    \frac{\mathscr{A}(-x)-\mathscr{A}(x)}{x}\;, & & x\ne0\;, \\
    -2\mathcal{L}\mathscr{A}(0)\;, & & x=0\;.
  \end{array}\right.
\end{equation}
Then $\mathscr{B}$ is entire and it has a removable singularity at $x=0$:
\begin{displaymath}
  \lim_{x\to0}\frac{\mathscr{A}(-x)-\mathscr{A}(x)}{x} = -2\mathscr{A}'(0) = -2\mathscr{A}(0)\mathcal{L}\mathscr{A}(0) = -2\mathcal{L}\mathscr{A}(0)\;.
\end{displaymath}
Hence~\cref{eq:instab_int_hr_a} becomes:
\begin{equation}\label{eq:instab_int_hr_b}
  h_r(\theta) = \mathscr{B}(0)+2\sum_{k=1}^{l'}\mathscr{B}(k)\cos{k\theta} = \sum_{k\in\mathbb{Z}}\mathscr{B}(k)e^{ik\theta}\;,
\end{equation}
which resembles the Fourier series sum of some $2\pi$-periodic function.

In the rest of this sub-section, the Fourier transform of a function $\mathscr{X}(x)$ is denoted $\hat{\mathscr{X}}(t)\eqdef\int_{-\infty}^\infty\mathscr{X}(x)\,e^{i2\pi xt}\,dx$.
As $\mathscr{C}^{m,n}(x)$ is essentially a reciprocal beta function, we may rewrite it in integral form~\cite[5.12.5]{FWJOlver:2010a}:
\begin{align*}
  \mathscr{C}^{m,n}(x) &= \frac{B(m+1,n+1)}{B(m+1+x,n+1-x)} = \frac{2^{m+n+1}m!n!}{\pi(m+n)!}\int_0^{\frac{\pi}{2}}\cos^{m+n}t\,\cos\left((2x+m-n)t\right)\,dt \\
  &= \frac{C_{m,n}}{\pi}\int_{-\frac{\pi}{2}}^{\frac{\pi}{2}}\cos^{m+n}t\,e^{-i(2x+m-n)t}\,dt
   = C_{m,n}\int_{-\frac{1}{2}}^{\frac{1}{2}}\cos^{m+n}(\pi t)\,e^{-i(2x+m-n)\pi t}\,dt\;,
\end{align*}
where:
\begin{equation}\label{eq:instab_int_cmn}
  C_{m,n} = \frac{2^{m+n}m!n!}{(m+n)!}\;.
\end{equation}
Therefore the Fourier transform of $\mathscr{C}^{m,n}(x)$ is:
\begin{equation}\label{eq:instab_int_c_ft}
  \hat{\mathscr{C}}^{m,n}(t) = C_{m,n}\cos^{m+n}(\pi t)e^{i(n-m)\pi t}\chi_{\left[-\frac{1}{2},\frac{1}{2}\right]}(t)\,,
\end{equation}
where $\chi_I(t)$ is the indicator function that equals $1$ if $t$ is in the interval $I$ and $0$ if otherwise.
It follows from the definition~\cref{eq:instab_int_a} that:
\begin{equation}\label{eq:instab_int_a_int}
  \mathscr{A}(x) = \int_{-1}^1(\hat{\mathscr{C}}^{l,r}\ast\hat{\mathscr{C}}^{l',r'})(t)\,e^{-i2\pi xt}\,dt\;,
\end{equation}
where $\ast$ is the convolution and we used the fact that $\hat{\mathscr{A}}(t) = (\hat{\mathscr{C}}^{l,r}\ast\hat{\mathscr{C}}^{l',r'})(t)$ is a continuously differentiable function on $\mathbb{R}$ and compactly supported on $[-1, 1]$.
Denoting the positive and negative parts of a real number $t$ by $t^+$ and $t^-$, respectively, by direct computation:
\begin{equation}\label{eq:instab_int_a_ft}
  \hat{\mathscr{A}}(t) = \chi_{[-1,1]}(t)\,A\,e^{i(r'-l')\pi t}\int_{t^+-\frac{1}{2}}^{t^-+\frac{1}{2}}\cos^{l+r}(\pi t')\,\cos^{l'+r'}(\pi(t-t'))\,e^{i(r-l-r'+l')\pi t'}\,dt'\;.
\end{equation}
where $A=C_{l,r}C_{l',r'}$.
By the definition of $\mathscr{B}$, its Fourier transform $\hat{\mathscr{B}}$ is an anti-derivative of $\mathscr{A}(-t)-\mathscr{A}(t)$; once we constructed an integral representation of $\mathscr{B}$, the Fourier series~\cref{eq:instab_int_hr_b} can be found equal to a periodization of $\hat{\mathscr{B}}$, more details are given next.

Note that $r-r', l-l'\in\{0,1\}$, thus $r-l-r'+l'\in\{-1,0,1\}$; we consider first $r-l-r'+l'=0$ that simplifies the integral, and then consider the other cases. 

\medskip

\noindent
{\it Case 1: $r-r'=l-l'$}, then we write $p=l'-r'=l-r=\frac{1}{2}(L-R)\ge0$:
\begin{align*}
  &\hat{\mathscr{A}}(t) = \chi_{[-1,1]}(t)\,A\,e^{-ip\pi t}\int_{t^+-\frac{1}{2}}^{t^-+\frac{1}{2}}\cos^{l+r}(\pi t')\,\cos^{l'+r'}(\pi(t-t'))\,dt'\;, \\
  \Rightarrow\quad
  &\hat{\mathscr{A}}(-t) - \hat{\mathscr{A}}(t) = 2i\chi_{[-1,1]}(t)A\sin(p\pi t)\int_{t^+-\frac{1}{2}}^{t^-+\frac{1}{2}}\cos^{l+r}(\pi t')\,\cos^{l'+r'}(\pi(t-t'))\,dt'\;.
\end{align*}
Define:
\begin{equation}\label{eq:instab_int_d}
  \mathscr{D}(t) = \chi_{[-1,1]}(t)\int_{t^+-\frac{1}{2}}^{t^-+\frac{1}{2}}\cos^{l+r}(\pi t')\,\cos^{l'+r'}(\pi(t-t'))\,dt'\;,
\end{equation}
which is the convolution of $\chi_{[-1/2,1/2]}(t)\cos^{l+r}(t)$ and $\chi_{[-1/2,1/2]}(t)\cos^{l'+r'}(t)$, both of which are even, positive on $(-1/2,1/2)$, and strictly decreasing on $[0,1/2]$, and they have continuous derivatives on the entire real line $\mathbb{R}$.
Therefore the properties of $\mathscr{D}$ include:
\begin{equation}\label{eq:instab_int_d_prop}
  \begin{array}{lll}
    \mathscr{D}\in C^1(\mathbb{R})\;; & \quad\mathscr{D}(t) = 0\;,\ \abs{t}\ge1\;; & \quad\mathscr{D}(t)=\mathscr{D}(-t)\;; \\
    \mathscr{D}(t)>0\;,\ \abs{t}<1\;; & \quad\mathscr{D}'(t)<0\;,\ 0<t<1\;.
  \end{array}
\end{equation}

Moving on with the construction of the Fourier transform of $\mathscr{B}(x)$, we define:
\begin{equation}\label{eq:instab_int_b_ft}
  \hat{\mathscr{B}}(t) = 2A\int_{\min(\abs{t},1)}^1\sin(p\pi t')\,\mathscr{D}(t')\,dt'\;,
\end{equation}
then it is easy to verify that $\hat{\mathscr{B}}$ is (1) continuously differentiable on $\mathbb{R}$, (2) compactly supported on $[-1,1]$, and (3) even; and we claim that~\cref{eq:instab_int_b_ft} is precisely the Fourier transform of $\mathscr{B}(x)$.
Indeed, one computes:
\begin{displaymath}
  \hat{\mathscr{B}}'(t) = -2A\sin(p\pi t)\mathscr{D}(t) = i\left[\hat{\mathscr{A}}(-t)-\hat{\mathscr{A}}(t)\right]\;.
\end{displaymath}
Taking the inverse Fourier transform of both sides, we obtain the inverse Fourier transform of $\hat{\mathscr{B}}(t)$ as $\frac{i\left[\mathscr{A}(-x)-\mathscr{A}(x)\right]}{ix} = \mathscr{B}(x)$.

Next we define a $1$-periodization of $\hat{\mathscr{B}}$ as:
\begin{equation}\label{eq:instab_int_b_per}
  \tilde{\mathscr{B}}(t) = \sum_{k=-\infty}^\infty\hat{\mathscr{B}}(t+k)\;,
\end{equation}
which is well defined as $\hat{\mathscr{B}}$ is compactly supported in $[-1,1]$; and it is easy to see that $\tilde{\mathscr{B}}$ is continuously differentiable on $\mathbb{R}$.
The complex Fourier coefficient for $k\in\mathbb{Z}$ of $\tilde{\mathscr{B}}$ is:
\begin{displaymath}
  \int_{-\frac{1}{2}}^{\frac{1}{2}}\tilde{\mathscr{B}}(t)e^{-2\pi ikt}\,dt = \int_{-1}^1\hat{\mathscr{B}}(t)e^{-2\pi ikt}\,dt = \mathscr{B}(k)\;;
\end{displaymath}
therefore by~\cref{eq:instab_int_hr_b} $h_r(\theta)$ is precisely $\tilde{\mathscr{B}}(t)$ with $t=\frac{\theta}{2\pi}$, which indicates:
\begin{equation}\label{eq:instab_int_hr_case1}
  h_r(\theta) = \tilde{\mathscr{B}}\left(\frac{\theta}{2\pi}\right) = \hat{\mathscr{B}}\left(\frac{\theta}{2\pi}\right)+\hat{\mathscr{B}}\left(\frac{\theta}{2\pi}-1\right)\;,\quad 0\le\theta\le2\pi\;.
\end{equation}

\medskip

\noindent
{\it Case 2: $r-r'=1, l-l'=0$}. 
It is clear that $l'>0$.
Let us define a auxiliary method with stencil $(l_0,r_0,l_0',r_0')=(l,r,l'-1,r')$, and use the subscript $_0$ to denote any quantity associated with this method.
Because the auxiliary method belongs to Case 1, one has $h_{r0}(\theta) = \hat{\mathscr{B}}_0(\theta/(2\pi))+\hat{\mathscr{B}}_0(\theta/(2\pi)-1)$ as given before.
To secure a connection between $h_r$ and $h_{r0}$, we start with $\mathscr{A}$ and $\mathscr{A}_0$:
\begin{displaymath}
  \mathscr{A}(x) = \mathscr{C}^{l,r}(x)\mathscr{C}^{l',r'}(x)\;,\quad
  \mathscr{A}_0(x) = \mathscr{C}^{l,r}(x)\mathscr{C}^{l'-1,r'}(x)\quad\Rightarrow\quad
  \mathscr{A}(x) = \frac{l'}{l'+x}\mathscr{A}_0(x)\;,
\end{displaymath}
therefore the $\mathscr{B}$-symbols satisfy:
\begin{displaymath}
  \mathscr{B}(x)-\mathscr{B}_0(x) = \frac{\mathscr{A}(-x)-\mathscr{A}(x)-\frac{l'-x}{l'}\mathscr{A}(-x)+\frac{l'+x}{l'}\mathscr{A}(x)}{x} = \frac{\mathscr{A}(-x)+\mathscr{A}(x)}{l'}\;.
\end{displaymath}
Thus the Fourier transform of $\mathscr{B}-\mathscr{B}_0$ is precisely $\hat{\mathscr{A}}(-t)+\hat{\mathscr{A}}(t)$, which is compactly supported on $[-1,1]$.
By defining a $1$-periodization of $\hat{\mathscr{A}}(-t)+\hat{\mathscr{A}}(t)$ like~\cref{eq:instab_int_b_per}, one computes:
\begin{equation}\label{eq:instab_int_hr_case2}
  h_r(\theta)-h_{r0}(\theta) = \sum_{k\in\mathbb{Z}}\frac{\mathscr{A}(k)+\mathscr{A}(-k)}{l'}e^{ik\theta} = \frac{2}{l'}\left[\oname{Re}\hat{\mathscr{A}}\left(\frac{\theta}{2\pi}\right)+\oname{Re}\hat{\mathscr{A}}\left(\frac{\theta}{2\pi}-1\right)\right]\;,
\end{equation}
where we used the fact that $\hat{\mathscr{A}}(t)$ is the complex conjugate of $\hat{\mathscr{A}}(-t)$.

\medskip

\noindent
{\it Case 3: $r-r'=0, l-l'=1$}. 
The strategy is similar as before, we consider an auxiliary method with the stencil $(l_0,r_0,l_0',r_0')=(l-1,r,l',r')$, which belongs to Case 1.
In this case:
\begin{displaymath}
  \mathscr{A}(x) = \mathscr{C}^{l,r}(x)\mathscr{C}^{l',r'}(x)\;,\quad
  \mathscr{A}_0(x) = \mathscr{C}^{l-1,r}(x)\mathscr{C}^{l',r'}(x)\quad\Rightarrow\quad
  \mathscr{A}(x) = \frac{l}{l+x}\mathscr{A}_0(x)\;,
\end{displaymath}
and consequently $\mathscr{B}(x)-\mathscr{B}_0(x)=(\mathscr{A}(-x)+\mathscr{A}(x))/l$, therefore:
\begin{equation}\label{eq:instab_int_hr_case3}
  h_r(\theta)-h_{r0}(\theta) = \frac{2}{l}\left[\oname{Re}\hat{\mathscr{A}}\left(\frac{\theta}{2\pi}\right)+\oname{Re}\hat{\mathscr{A}}\left(\frac{\theta}{2\pi}-1\right)\right]\;.
\end{equation}

\medskip

Summarizing all three cases, especially~\cref{eq:instab_int_hr_case1}, \cref{eq:instab_int_hr_case2}, and~\cref{eq:instab_int_hr_case3}, and using the integrals~\cref{eq:instab_int_a_ft} and~\cref{eq:instab_int_b_ft}, we obtain the following lemma on the integral representation of $h_r(\theta)$.
\begin{lemma}\label{lm:instab_int}
  Let the HV scheme be given by the stencil $(l,r,l',r')$ such that $L\ge R$, then the real part of $H(e^{i\theta})$ has the following integral form:

  Case 1: If $r-r'=l-l'$, for all $0\le\theta\le2\pi$,
  \begin{align}
    \notag
    h_r(\theta) &= 2A\left(\int_{\frac{\theta}{2\pi}}^1+\int_{1-\frac{\theta}{2\pi}}^1\right)\sin(p\pi t)\left[\int_{t-\frac{1}{2}}^{\frac{1}{2}}\cos^{l+r}(\pi t')\,\cos^{l'+r'}(\pi(t-t'))\,dt'\right]\,dt\;. \\
    \label{eq:instab_int_hr_c1}
    &= \frac{A}{\pi}\left(\int_\theta^{2\pi}+\int_{2\pi-\theta}^{2\pi}\right)\mathscr{D}\left(\frac{\theta'}{2\pi}\right)\,\sin\frac{p\theta'}{2}\,d\theta'\;.
  \end{align}

  Case 2: If $r-r'=1$ and $l-l'=0$, with $p'=p-1$ then for all $0\le\theta\le2\pi$,
  \begin{align}
    \notag
    h_r(\theta) =&\ \frac{(L+R)A}{l'}\left(\int_{\frac{\theta}{2\pi}}^1+\int_{1-\frac{\theta}{2\pi}}^1\right)\sin(p'\pi t)\left[\int_{t-\frac{1}{2}}^{\frac{1}{2}}\cos^{l+r}(\pi t')\,\cos^{l'+r'-1}(\pi(t-t'))\,dt'\right]\,dt + \\
    \label{eq:instab_int_hr_c2}
                 &\ \frac{2A}{l'}\int_{\frac{\theta}{2\pi}-\frac{1}{2}}^{\frac{1}{2}}\cos^{l+r}(\pi t')\,\cos^{l'+r'}\left(\frac{\theta}{2}-\pi t'\right)\,\cos\left(\frac{p\theta}{2}-\pi t'\right)\,dt' + \\
    \notag
                 &\ \frac{2A}{l'}\int_{\frac{1}{2}-\frac{\theta}{2\pi}}^{\frac{1}{2}}\cos^{l+r}(\pi t')\,\cos^{l'+r'}\left(\frac{\theta}{2}+\pi t'\right)\,\cos\left(\frac{p\theta}{2}+\pi t'\right)\,dt'\;.
  \end{align}

  Case 3: If $r-r'=0$ and $l-l'=1$, for all $0\le\theta\le2\pi$,
  \begin{align}
    \notag
    h_r(\theta) =&\ \frac{(L+R)A}{l}\left(\int_{\frac{\theta}{2\pi}}^1+\int_{1-\frac{\theta}{2\pi}}^1\right)\sin(p\pi t)\left[\int_{t-\frac{1}{2}}^{\frac{1}{2}}\cos^{l+r-1}(\pi t')\,\cos^{l'+r'}(\pi(t-t'))\,dt'\right]\,dt + \\
    \label{eq:instab_int_hr_c3}
                 &\ \frac{2A}{l}\int_{\frac{\theta}{2\pi}-\frac{1}{2}}^{\frac{1}{2}}\cos^{l+r}(\pi t')\,\cos^{l'+r'}\left(\frac{\theta}{2}-\pi t'\right)\,\cos\left(\frac{p\theta}{2}+\pi t'\right)\,dt' + \\
    \notag
                 &\ \frac{2A}{l}\int_{\frac{1}{2}-\frac{\theta}{2\pi}}^{\frac{1}{2}}\cos^{l+r}(\pi t')\,\cos^{l'+r'}\left(\frac{\theta}{2}+\pi t'\right)\,\cos\left(\frac{p\theta}{2}-\pi t'\right)\,dt'\;.
  \end{align}
  Here in all cases $A=C_{l,r}C_{l',r'}=\frac{2^{l+r+l'+r'}l!r!l'!r'!}{(l+r)!(l'+r')!}$ as defined before and $p=l'-r'$.
\end{lemma}

\subsection{HV methods with $L-R>3$}
\label{sec:instab_h}
Motivated by the previous work~\cite{XZeng:2026a}, the smallest value of $h_r(\theta)$ is achieved around $\theta^\ast=2\pi/((L-R)/2)$, at least when $L-R$ is large.
We shall adopt a similar strategy here and prove whenever $L-R>3$, $h_r(\theta^\ast)<0$ for some carefully chosen $\theta^\ast$.
The proof utilizes~\cref{lm:instab_int}, and we again distinguish among three cases.

\medskip

\noindent
{\it Case 1: $r-r'=l-l'$}, then $p=l'-r'=l-r=\frac{1}{2}(L-R)\ge2$.
Let $\theta^\ast=\frac{2\pi}{l'-r'}=\frac{2\pi}{p}$, by~\cref{eq:instab_int_hr_c1} and utilizing $\mathscr{D}(t)$ as defined in~\cref{eq:instab_int_d}, we can compute:
\begin{align*}
  h_r(\theta^\ast) &= \frac{A}{\pi}\left(\int_{\theta^\ast}^{2\pi}+\int_{2\pi-\theta^\ast}^{2\pi}\right)\mathscr{D}\left(\frac{\theta'}{2\pi}\right)\sin\frac{p\theta'}{2}d\theta' \\
  &= \frac{A}{\pi}\sum_{k=1}^{p-1}\int_{k\theta^\ast}^{(k+1)\theta^\ast}\mathscr{D}\left(\frac{\theta'}{2\pi}\right)\sin\frac{p\theta'}{2}d\theta'
 + \frac{A}{\pi}\int_{(p-1)\theta^\ast}^{p\theta^\ast}\mathscr{D}\left(\frac{\theta'}{2\pi}\right)\sin\frac{p\theta'}{2}d\theta'\;.
\end{align*}
Define for $1\le k\le p-1$:
\begin{displaymath}
  I_k = (-1)^k\int_{k\theta^\ast}^{(k+1)\theta^\ast}\mathscr{D}\left(\frac{\theta'}{2\pi}\right)\sin\frac{p\theta'}{2}d\theta'
  = \int_0^{\theta^\ast}\mathscr{D}\left(\frac{k}{p}+\frac{\theta'}{2\pi}\right)\sin\frac{p\theta'}{2}\,d\theta' > 0\;.
\end{displaymath}
Because $\mathscr{D}$ decreases on $[0, 1]$, there is $I_1>I_2>\cdots>I_{p-1}>0$.
Therefore if $p$ is even:
\begin{displaymath}
  h_r(\theta^\ast) = \frac{A}{\pi}\left[-(I_1-I_2) - (I_3-I_4) - \cdots - (I_{p-3}-I_{p-2}) - 2I_{p-1}\right] < 0\;,
\end{displaymath}
which completes the proof.

Now suppose $p$ is odd, then $p\ge3$, and we aim at showing $I_{p-2}\ge2 I_{p-1}$, in which case:
\begin{displaymath}
  h_r(\theta^\ast) = \frac{A}{\pi}\left[-(I_1-I_2) - (I_3-I_4) - \cdots - (I_{p-2}-2I_{p-1})\right] < 0\;. 
\end{displaymath}
A sufficient condition for $I_{p-2}\ge2 I_{p-1}$ is $\mathscr{D}\left(\frac{p-2}{p}+\frac{\theta'}{2\pi}\right)>2\mathscr{D}\left(\frac{p-1}{p}+\frac{\theta'}{2\pi}\right)$.
Defining $t_1 = 1-\left(\frac{p-2}{p}+\frac{\theta'}{2\pi}\right) = \frac{2}{p}-\frac{\theta'}{2\pi}$ and $t_2 = 1-\left(\frac{p-1}{p}+\frac{\theta'}{2\pi}\right) = \frac{1}{p}-\frac{\theta'}{2\pi}$, we obtain:
\begin{align}
  \label{eq:instab_h_d1}
  \mathscr{D}\left(\frac{p-2}{p}+\frac{\theta'}{2\pi}\right) = \mathscr{D}(1-t_1)
  &= \frac{t_1}{\pi}\int_0^\pi\sin^{l+r}(t_1\theta'')\,\sin^{l'+r'}(t_1(\pi-\theta''))\,d\theta''\;, \\
  \label{eq:instab_h_d2}
  \mathscr{D}\left(\frac{p-1}{p}+\frac{\theta'}{2\pi}\right) = \mathscr{D}(1-t_2)
  &= \frac{t_2}{\pi}\int_0^\pi\sin^{l+r}(t_2\theta'')\,\sin^{l'+r'}(t_2(\pi-\theta''))\,d\theta''\;.
\end{align}
For all $0<\theta''<\pi$, there is $0<t_2\theta''<t_1\theta''$ and:
\begin{displaymath}
  (\pi-t_2\theta'')-t_1\theta'' = \pi-(t_1+t_2)\theta'' = \pi-\left(\frac{3}{p}-\frac{\theta'}{\pi}\right)\theta'' > \pi-\frac{3}{p}\pi \ge 0\;,
\end{displaymath}
thus $0<t_2\theta''<t_1\theta''<\pi-t_2\theta''$, which indicates $\sin(t_1\theta'')>\sin(t_2\theta'')>0$.
Changing $\theta''$ to $\pi-\theta''$, which is again between $0$ and $\pi$, one immediately gets $\sin(t_1(\pi-\theta''))>\sin(t_2(\pi-\theta''))>0$.
Hence the integral of~\cref{eq:instab_h_d1} is strictly larger than that of~\cref{eq:instab_h_d2} and both are positive; therefore $\mathscr{D}(1-t_1)>2\mathscr{D}(1-t_2)$ immediagely follows from $t_1>\frac{2}{p}-\frac{2\theta'}{2\pi}=2t_2$ for all $0<\theta'<\theta^\ast$, which completes the proof of $I_{p-2}>2I_{p-1}$.

\medskip

\noindent
{\it Case 2: $r-r'=1, l-l'=0$}, then $p=l'-r'=\frac{1}{2}(L-R+1)\ge3$. 
Let all quantities related to the stencil $(l_0,r_0,l_0',r_0')=(l,r,l'-1,r')$ be designated with the subscript $_0$ as before and $L_0-R_0=L-R-1=2p-2>3$, then by the previous case one has:
\begin{displaymath}
  h_{r0}(\theta^\ast) < 0 \;,\quad\textrm{ where }\quad
  \theta^\ast = \frac{2\pi}{l_0'-r_0'} = \frac{2\pi}{p-1}\;.
\end{displaymath}
Noting that $\frac{p\theta^\ast}{2}=\pi+\frac{\theta^\ast}{2}$, by~\cref{eq:instab_int_hr_c2} we have:
\begin{align*}
  (h_r(\theta^\ast)-h_{r0}(\theta^\ast))/(2A/l') =&\ -\int_{\frac{\theta^\ast}{2\pi}-\frac{1}{2}}^{\frac{1}{2}}\cos^{l+r}(\pi t')\,\cos^{l'+r'+1}\left(\frac{\theta^\ast}{2}-\pi t'\right)\,dt' \\
  &\ - \int_{\frac{1}{2}-\frac{\theta^\ast}{2\pi}}^{\frac{1}{2}}\cos^{l+r}(\pi t')\,\cos^{l'+r'+1}\left(\frac{\theta^\ast}{2}+\pi t'\right)\,dt'\;.
\end{align*}
Let $\mathscr{D}_1(t)$ be defined by~\cref{eq:instab_int_d} that corresponds to the stencil $(l,r,l',r'+1)$, which clearly belongs to Case 1, then the last equality can be written as:
\begin{displaymath}
  \frac{h_r(\theta^\ast)-h_{r0}(\theta^\ast)}{2A/l'} = -\mathscr{D}_1\left(\frac{\theta^\ast}{2\pi}\right)+(-1)^p\mathscr{D}_1\left(1-\frac{\theta^\ast}{2\pi}\right)\;.
\end{displaymath}
We see that the right hand side is negative because by~\cref{eq:instab_int_d_prop}, $\mathscr{D}_1(t)$ is strictly decreasing on $[0,1]$ and it is clear: $\frac{\theta^\ast}{2\pi}-\left(1-\frac{\theta^\ast}{2\pi}\right) = \frac{2}{p-1}-1\le0$.
Therefore, $h_r(\theta^\ast)<h_{r0}(\theta^\ast)<0$.

\medskip

\noindent
{\it Case 3: $r-r'=0, l-l'=1$}, then $p=l'-r'=\frac{1}{2}(L-R-1)\ge2$.
Similar as in~\cref{sec:instab_int}, we designate all quantities related to the stencil $(l_0,r_0,l_0',r_0')=(l-1,r,l',r')$ by the subscript $_0$, then $L_0-R_0=L-R-1=2p>3$ and it belongs to Case 1, and one has:
\begin{displaymath}
  h_{r0}(\theta^\ast) < 0 \;,\quad\textrm{ where }\quad
  \theta^\ast = \frac{2\pi}{l_0'-r_0'} = \frac{2\pi}{p}\;.
\end{displaymath}
By~\cref{eq:instab_int_hr_c3}, one obtains:
\begin{align*}
  (h_r(\theta^\ast)-h_{r0}(\theta^\ast))/(2A/l) =&\ -\int_{\frac{\theta^\ast}{2\pi}-\frac{1}{2}}^{\frac{1}{2}}\cos^{l+r+1}(\pi t')\,\cos^{l'+r'}\left(\frac{\theta^\ast}{2}-\pi t'\right)\,dt' \\
  &\ - \int_{\frac{1}{2}-\frac{\theta^\ast}{2\pi}}^{\frac{1}{2}}\cos^{l+r+1}(\pi t')\,\cos^{l'+r'}\left(\frac{\theta^\ast}{2}+\pi t'\right)\,dt'\;,
\end{align*}
which is precisely $-\mathscr{D}_1\left(\frac{\theta^\ast}{2\pi}\right)-(-1)^p\mathscr{D}_1\left(1-\frac{\theta^\ast}{2\pi}\right)<0$.
Here $\mathscr{D}_1$ is defined by~\cref{eq:instab_int_d} using the stencil $(l,r+1,l',r')$ that belongs to Case 1, and the negativity comes from the fact that $\mathscr{D}_1$ strictly decreases on $[0,1]$ and $\frac{\theta^\ast}{2\pi}-\left(1-\frac{\theta^\ast}{2\pi}\right) = \frac{2}{p}-1\le0$.
Hence $h_r(\theta^\ast)<h_{r0}(\theta^\ast)<0$.

\medskip

\noindent
In summary, $h_r(\theta^\ast)<0$ at the chosen $\theta^\ast$ in all cases, showing that the method is unstable.

\subsection{HV methods with $L-R=3$}
\label{sec:instab_f}
By~\cref{tb:prelim_stab}, the first stability condition~\cref{eq:prelim_stab_cond_h} likely holds when $L-R=3$; thus we focus on the second condition~\cref{eq:prelim_stab_cond_g} when $L-R=3$ and assume $R\ge1$\footnote{The method with $R=0$ and $L=3$ is actually stable, see~\cite{XZeng:2019a}.}.
In this case, the stencil is given by either (a) $(l,r,l',r')=(r+2,r,r+1,r)$ or (b) $(l,r,l',r')=(r+1,r,r+1,r-1)$, where in both cases $r\ge1$.

The strategy is to show $q(\theta)>0$ when $\abs{\theta}$ is small, where $q(\theta)$ is given by~\cref{eq:prelim_stab_cond_equiv}.
The difficulty lies in the fact that the leading terms of the Taylor series expansion of $q(\theta)$ at $\theta=0$ are zero and we need to seek the first non-zero term.

Inspired by a technique that we have used in studying the central HV methods for diffusion equations, we compute the Laurent series expansion of:
\begin{equation}\label{eq:instab_f_p}
  P(x) = \frac{(-1)^{R-1}2l!l'!r!r'!}{x\,\prod_{k=-l}^r(x-k)\,\prod_{k=-l'}^{r'}(x-k)}
\end{equation}
in two different ways.
First, we find the partial fraction expansion of $P(x)$:
\begin{equation}\label{eq:instab_f_p_gen}
  P(x) = \frac{P_{0,3}}{x^3} + \frac{P_{0,2}}{x^2} + \frac{P_{0,1}}{x} + \sum_{k=-l,k\ne0}^r\frac{P_{k,1}}{x-k} + \sum_{k=-l',k\ne0}^{r'}\frac{P_{k,2}}{(x-k)^2}\;.
\end{equation}
Then it is straightforward to compute:
\begin{align*}
  &P_{0,3} = \lim_{x\to0}x^3P(x) = -2\;,\quad 
  P_{0,2} = \frac{d}{dx}\left[x^3P(x)\right]\Big|_{x=0} = 2(H_{l'}-H_{r'}+H_l-H_r) = \beta_0\;, \\
  &P_{k,2} = \lim_{x\to k}(x-k)^2P(x) = -\frac{2}{k}C^{l,r}_kC^{l',r'}_k = \beta_k\;,\quad \forall -l'\le k\le r',\ k\ne0\;;
\end{align*}
as for $P_{k,1}$, using logarithmic derivative one computes for $-l'\le k\le r'$ and $k\ne0$:
\begin{displaymath}
  P_{k,1} = \frac{d}{dx}\left[(x-k)^2P(x)\right]\Big|_{x=k} = -P_{k,2}\left(\frac{1}{k}+\zeta^{l,r}_k+\zeta^{l',r'}_k\right) = t_k\;,
\end{displaymath}
and as a matter of fact using the same technique one can show $P_{k,1}=t_k$ for all $-l\le k\le r$ and $k\ne 0$, if it is not already included in $[-l',r']$.
To compute $P_{0,1}$, one notices that the residual of $P(x)$ at $\infty$ is zero, thus one has:
\begin{displaymath}
  \sum_{k=-l}^r\oname{Res}(P,k) = 0\;,
\end{displaymath}
where $\oname{Res}(P,k)$ is the residual of $P$ at $x=k$, or $\oname{Res}(P,k) = P_{k,1}$.
Thus:
\begin{displaymath}
  P_{0,1} = -\sum_{k=-l,k\ne0}^rP_{0,k} = -\sum_{k=-l,k\ne0}^rt_k\;;
\end{displaymath}
therefore we obtain the partial fraction expansion:
\begin{equation}\label{eq:instab_f_p_pf}
  P(x) = -\frac{2}{x^3} + \frac{\beta_0}{x^2} + \sum_{k=-l,k\ne0}^r\left(\frac{t_k}{x-k}-\frac{t_k}{x}\right) + \sum_{k=-l',k\ne0}^{r'}\frac{\beta_k}{(x-k)^2}\;.
\end{equation}
Setting $x=1/z$, one gets:
\begin{align}
  \notag
  P\left(\frac{1}{z}\right) &= -2z^3+\beta_0z^2+\sum_{k=-l,k\ne0}^rt_kz\sum_{n=1}^\infty(kz)^n+\sum_{k=-l',k\ne0}^{r'}\beta_kz^2\sum_{n=0}^\infty(n+1)(kz)^n \\
  \label{eq:instab_f_p_inv}
  &= \left([\beta]_0+[t]_1\right)z^2 + \left(-2+2[\beta]_1+[t]_2\right)z^3 + \sum_{n=4}^\infty\left((n-1)[\beta]_{n-2}+[t]_{n-1}\right)z^n\;,
\end{align}
where $[\beta]_n$ and $[t]_n$ are the $n$-th moments of the $\beta$- and the $t$-sequences, respectively:
\begin{equation}\label{eq:instab_f_coef_pow}
  [\beta]_n \eqdef \sum_{k=-l'}^{r'}k^n\beta_k\;,\quad
  [t]_n \eqdef \sum_{k=-l,k\ne0}^rk^nt_k\;.
\end{equation}

The formal power series of $P(1/z)$ can be computed directly in another way:
\begin{align*}
  P\left(\frac{1}{z}\right) &= \frac{(-1)^{R-1}2l!r!l'!r'!z^{L+R+3}}{\prod_{k=-l,k\ne0}^r(1-kz)\prod_{k=-l',k\ne0}^{r'}(1-kz)} \\
  &= (-1)^{R-1}2l!r!l'!r'!z^{L+R+3}\left(1-C_1z\right) + O(z^{L+R+5})\;,
\end{align*}
where $C_1=\sum_{k=-l'}^{r'}k+\sum_{k=-l}^rk$.
Comparing it with~\cref{eq:instab_f_coef_pow}, we obtain:
\begin{align*}
  &[\beta]_0+[t]_1 = 0\;;\qquad-2+2[\beta]_1+[t]_2 = 0\;; \\ 
  &(n-1)[\beta]_{n-2}+[t]_{n-1} = 0\;,\quad 4\le n\le L+R+2\;; \\
  &(L+R+2)[\beta]_{L+R+1}+[t]_{L+R+2}=(-1)^{R-1}2l!r!l'!r'!\;; \\
  &(L+R+3)[\beta]_{L+R+2}+[t]_{L+R+3}=(-1)^R2l!r!l'!r'!C_1\;.
\end{align*}

Now let us compute the power series expansion of $H(e^{i\theta})$ and $F(e^{i\theta})$ around $\theta=0$:
\begin{align*}
  H(e^{i\theta}) &= \sum_{k=-l'}^{r'}\beta_ke^{ik\theta} = \sum_{n=0}^\infty\sum_{k=-l'}^{r'}(ik)^n\beta_k\frac{\theta^n}{n!} = \sum_{n=0}^\infty\frac{i^n}{n!}[\beta]_n\theta^n\;, \\
  F(e^{i\theta}) &= \sum_{k=-l,k\ne0}^rt_k(e^{ik\theta}-1) = \sum_{n=1}^\infty\sum_{k=-l,k\ne0}^{r}(ik)^nt_k\frac{\theta^n}{n!} = \sum_{n=1}^\infty\frac{i^n}{n!}[t]_n\theta^n\;.
\end{align*}
Therefore:
\begin{align*}
  \theta^2+i\theta H+F &= \theta^2+\sum_{n=1}^\infty\frac{i^n}{n!}(n[\beta]_{n-1})\theta^n+\sum_{n=1}^\infty\frac{i^n}{n!}[t]_n\theta^n \\
  &= -i\frac{2l!r!l'!r'!}{(L+R+2)!}\theta^{L+R+2}-\frac{2l!r!l'!r'!C_1}{(L+R+3)!}\theta^{L+R+3}+(\theta^{L+R+4})\;.
\end{align*}
Extracting the real part and the imaginary part, one obtains:
\begin{align*}
  \theta^2-\theta h_i+f_r = -C_r\theta^{L+R+3} + \varepsilon_r\;,\quad
  \theta h_r + f_i = -C_i\theta^{L+R+2}+ \varepsilon_i\;,
\end{align*}
where $C_i=\frac{2l!r!l'!r'!}{(L+R+2)!}$, $C_r=\frac{2l!r!l'!r'!C_1}{(L+R+3)!}$, $\varepsilon_r=O(\theta^{L+R+4})$, and $\varepsilon_i=O(\theta^{L+R+4})$. 
Additionally, the Taylor series of $h_r$ and $h_i$ are:
\begin{displaymath}
  h_r(\theta) = [\beta]_0 - \frac{1}{2}[\beta]_2\theta^2 + O(\theta^4)\;,\qquad
  h_i(\theta) = [\beta]_1\theta - \frac{1}{6}[\beta]_3\theta^3 + O(\theta^5)\;,
\end{displaymath}
which allow us to compute:
\begin{align*}
  q(\theta) &= h_r(h_rf_r+h_if_i) + f_i^2 \\
  &= -C_rh_r^2\theta^{L+R+3}-C_ih_ih_r\theta^{L+R+2}+2C_ih_r\theta^{L+R+3} + O(\theta^{L+R+4}) \\
  &= \frac{2l!r!l'!r'![\beta]_0}{(L+R+2)!}\left(2-[\beta]_1-\frac{C_1}{L+R+3}[\beta]_0\right)\theta^{2(R+3)} + O(\theta^{2R+7})\;.
\end{align*}
By~\cite[Proposition 5]{XZeng:2019a}, $L>R$ indicates $[\beta]_0>0$, thus it remains to show that:
\begin{equation}\label{eq:instab_f_leading}
  C_{\beta}\eqdef 2-[\beta]_1-\frac{C_1}{L+R+3}[\beta]_0 > 0\;.
\end{equation}
Below we estimate the two moments of the $\beta$-sequences and start with $[\beta]_1$:
\begin{align*}
  2-[\beta]_1 &= 2-\sum_{k=-l',k\ne0}^{r'}k\left(-\frac{2}{k}C^{l,r}_kC^{l',r'}_k\right) = 2\sum_{k=-l'}^{r'}C^{l,r}_kC^{l',r'}_k\;;
\end{align*}
and for the term involving $[\beta]_0$:
\begin{align*}
  \frac{C_1}{L+R+3}[\beta]_0 = \frac{C_1}{L+R+3}\left(2(\zeta^{l,r}_0+\zeta^{l',r'}_0)-\sum_{k=-l',k\ne0}^{r'}\frac{2}{k}C^{l,r}_kC^{l',r'}_k\right)\;. 
\end{align*}
Therefore:
\begin{displaymath}
  C_\beta = 2-\frac{2C_1(\zeta_0^{l,r}+\zeta_0^{l',r'})}{L+R+3} + \sum_{k=-l',k\ne0}^{r'}\left(2-\frac{2C_1}{k(L+R+3)}\right)C^{l,r}_kC^{l',r'}_k\;,
\end{displaymath}
which will be discussed separately for the two cases.

\medskip

\noindent
{\it Case 1: $(l,r,l',r')=(r+2,r,r+1,r),\ r\ge1$}.
One gets $C_1=3r+4$ and:
\begin{displaymath}
  C_\beta = 2-\frac{2(3r+4)\left(\frac{1}{r+2}+\frac{2}{r+1}\right)}{4r+6} + \sum_{k=-r-1,k\ne0}^{r}\left(2-\frac{2(3r+4)}{k(4r+6)}\right)C^{r+2,r}_kC^{r+1,r}_k\;,
\end{displaymath}
which is always positive as for all $r\ge1$:
\begin{align*}
  &2-\frac{2(3r+4)\left(\frac{1}{r+2}+\frac{2}{r+1}\right)}{4r+6} = \frac{4r^3+9r^2-r-8}{(2r+3)(r+1)(r+2)} > 0\;, \\
  &2-\frac{2(3r+4)}{k(4r+6)} > 2 > 0\;,\ \forall k < 0\;;\quad
  2-\frac{2(3r+4)}{k(4r+6)} = \frac{(4k-3)r + (6k-4)}{k(2r+3)} > 0 \;,\ \forall k > 0\;.
\end{align*}

\medskip

\noindent
{\it Case 2: $(l,r,l',r')=(r+1,r,r+1,r-1),\ r\ge1$}.
One gets $C_1=3r+2$ and subsequently:
\begin{displaymath}
  C_\beta = 2-\frac{2(3r+2)\left(\frac{2}{r+1}+\frac{1}{r}\right)}{4r+4} + \sum_{k=-r-1,k\ne0}^{r-1}\left(2-\frac{2(3r+2)}{k(4r+4)}\right)C^{r+1,r}_kC^{r+1,r-1}_k\;.
\end{displaymath}
When $r\ge2$, all terms are positive because:
\begin{align*}
  &2-\frac{2(3r+2)\left(\frac{2}{r+1}+\frac{1}{r}\right)}{4r+4} = \frac{4r^3-r^2-5r-2}{2r(r+1)^2} > 0\;, \\
  &2-\frac{2(3r+2)}{k(4r+4)} > 2 > 0\;,\ \forall k < 0\;;\quad
  2-\frac{2(3r+2)}{k(4r+4)} = \frac{(4k-3)r + (4k-2)}{k(2r+2)} > 0 \;,\ \forall k > 0\;.
\end{align*}
The only issue is that when $r=1$, the first term is negative; but in this case, one has:
\begin{displaymath}
  C_\beta = 2-[\beta]_1-\frac{C_1[\beta]_0}{L+R+3} = -\frac{1}{2} + \sum_{k=-2}^{-1}\left(2-\frac{5}{4k}\right)C^{2,1}_kC^{2,0}_k = \frac{55}{8} > 0\;,
\end{displaymath}
which completes the proof.

\section{Stability proof when $0<L-R\le2$}
\label{sec:stab}
In this section we prove all HV discretizations with $L-R=1$ or $L-R=2$ are stable, that is: $(l,r,l',r')=(r+1,r,r,r), r\ge0$, $(l,r,l',r')=(r,r,r,r-1), r\ge1$, $(l,r,l',r')=(r+1,r,r+1,r), r\ge0$, or $(l,r,l',r')=(r+1,r,r,r-1), r\ge1$.

\subsection{Proving~\cref{eq:prelim_stab_cond_h}}
\label{sec:stab_h}
The proof utilizes the integral form of $h_r(\theta)$ given by~\cref{lm:instab_int}.
Let us begin with the case $r-r'=l-l'$, which includes the stencils $(r+1,r,r+1,r)$ and $(r+1,r,r,r-1)$.
By~\cref{eq:instab_int_hr_c1} (note $p=l'-r'=1$) and~\cref{eq:instab_int_d}:
\begin{displaymath}
  h_r(\theta) = \frac{A}{\pi}\left(\int_\theta^{2\pi}+\int_{2\pi-\theta}^{2\pi}\right)\mathscr{D}\left(\frac{\theta'}{2\pi}\right)\,\sin\frac{\theta'}{2}\,d\theta'\;.
\end{displaymath}
It is thus clear that $h_r(\theta)>0$ for all $\theta$ as the integrand is positive at all $\theta'\in(0, 2\pi)$. 

Taking the derivative one has:
\begin{displaymath}
  h_r'(\theta) = \frac{A}{\pi}\sin\frac{\theta}{2}\left[\mathscr{D}\left(1-\frac{\theta}{2\pi}\right)-\mathscr{D}\left(\frac{\theta}{2\pi}\right)\right]\;. 
\end{displaymath}
Because $\mathscr{D}$ strictly decreases on $[0, 1]$, $h_r'(\theta) < 0$ for $0<\theta\le\pi$.
This result will be used in~\cref{sec:stab_f} to prove~\cref{eq:prelim_stab_cond_g}.

\medskip

When $r-r'=1$ and $l-l'=0$, i.e., $(l,r,l',r')=(r,r,r,r-1)$ with $r\ge1$.
Then by~\cref{eq:instab_int_hr_c2}:
\begin{displaymath}
  \frac{h_r(\theta)-h_{r0}(\theta)}{2A/r} = 
    \int_{\frac{\theta}{2\pi}-\frac{1}{2}}^{\frac{1}{2}}\cos^{2r}(\pi t')\,\cos^{2r}\frac{\theta-2\pi t'}{2}\,dt' + 
    \int_{\frac{1}{2}-\frac{\theta}{2\pi}}^{\frac{1}{2}}\cos^{2r}(\pi t')\,\cos^{2r}\frac{\theta+2\pi t'}{2}\,dt'\,.
\end{displaymath}
Here $h_{r0}$ corresponds to the stencil $(l_0,r_0,l_0',r_0')=(r,r,r-1,r-1)$, which is symmetric and it is easy to verify $h_{r0}\equiv0$; thus: $h_r(\theta) > h_{r0}(\theta) = 0$ for all $0<\theta\le\pi$.

Taking the derivative yields:
\begin{align*}
  h_r'(\theta) =&\ - 2A\int_{\frac{\theta}{2\pi}-\frac{1}{2}}^{\frac{1}{2}}\cos^{2r}(\pi t')\,\cos^{2r-1}\frac{\theta-2\pi t'}{2}\,\sin\frac{\theta-2\pi t'}{2}\,dt' - \\
    &\ 2A\int_{\frac{1}{2}-\frac{\theta}{2\pi}}^{\frac{1}{2}}\cos^{2r}(\pi t')\,\cos^{2r-1}\frac{\theta+2\pi t'}{2}\,\sin\frac{\theta+2\pi t'}{2}\,dt' \\
  =&\ -2A\int_{-\pi}^{\pi}\cos^{2r}\frac{\theta'}{2}\,\cos^{2r-1}\frac{\theta-\theta'}{2}\,\sin\frac{\theta-\theta'}{2}\,d\theta'\;.
\end{align*}
Denoting the last integral by $\mathcal{I}(\theta)$, then using integration-by-part:
\begin{align*}
  \mathcal{I}(\theta)
  &= \frac{1}{r}\cos^{2r}\frac{\theta'}{2}\,\cos^{2r}\frac{\theta-\theta'}{2}\Big|^{\pi}_{-\pi} 
  + \int_{-\pi}^{\pi}\cos^{2r}\frac{\theta-\theta'}{2}\,\cos^{2r-1}\frac{\theta'}{2}\,\sin\frac{\theta'}{2}\,d\theta' \\
  &= \int_0^\pi\cos^{2r-1}\frac{\theta'}{2}\,\sin\frac{\theta'}{2}\,\left(\cos^{2r}\frac{\theta-\theta'}{2}-\cos^{2r}\frac{\theta+\theta'}{2}\right)\,d\theta' > 0\;,\quad0<\theta<\pi\;,
\end{align*}
where the last inequality is due to $\abs{(\theta-\theta')/2}<(\theta+\theta')/2<\pi$ and $\abs{(\theta-\theta')/2}+(\theta+\theta')/2<\pi$.
Therefore, we obtain again that $h_r'(\theta)<0$ for $0<\theta<\pi$.

\medskip

Lastly, suppose $r-r'=0$ and $l-l'=1$ or $(l,r,l',r')=(r+1,r,r,r)$, $r\ge0$.
The case with $r=0$ is trivial and we assume $r\ge1$; by~\cref{eq:instab_int_hr_c3} there is:
\begin{align*}
  \frac{h_r(\theta)}{2A/l} = 
    \int_{\frac{\theta}{2\pi}-\frac{1}{2}}^{\frac{1}{2}}\cos^{2r+2}(\pi t')\,\cos^{2r}\frac{\theta-2\pi t'}{2}\,dt' + 
    \int_{\frac{1}{2}-\frac{\theta}{2\pi}}^{\frac{1}{2}}\cos^{2r+2}(\pi t')\,\cos^{2r}\frac{\theta+2\pi t'}{2}\,dt'>0\;.
\end{align*}
Here we used $h_{r0}(\theta)\equiv0$ corresponding to the stencil $(l_0,r_0,l_0',r_0')=(l-1,r,l',r')=(r,r,r,r)$.
By a parallel argument to the previous one, we can show $h_r'(\theta)<0, 0<\theta<\pi$.

\subsection{Proving~\cref{eq:prelim_stab_cond_g}}
\label{sec:stab_f}
We shall discuss the four stencils separately.

\noindent
{\it Case 1: $(l,r,l',r')=(r+1,r,r,r)$}, and the $\beta$-coefficients are:
\begin{align*}
  \beta_0 &= 2(\zeta_0^{r+1,r}+\zeta_0^{r,r}) = \frac{2}{r+1}\;, \\
  \beta_k &= -\frac{2}{k}C^{r+1,r}_kC^{r,r}_k = \left(\frac{2}{r+1}-\frac{2}{k}\right)C^{r+1,r+1}_kC^{r,r}_k\;,\quad -r\le k\le r\;, k\ne0\;.
\end{align*}
We adopt the convention that $\beta_k=0$ if $\abs{k}>r$.
Writing $Z_k = C^{r+1,r+1}_kC^{r,r}_k$ for $-r\le k\le r$ and $Z_k=0$ if $\abs{k}>r$, which satisfies $Z_k=Z_{-k}$, there is the even-odd splitting:
\begin{displaymath}
  \beta_0 = \frac{2}{r+1}Z_0\;;\quad
  \beta_k = \frac{2}{r+1}Z_k-\frac{2}{k}Z_k\;,\quad \forall k\ne0\;.
\end{displaymath}
The real and imaginary parts of $H(e^{i\theta})$ are:
\begin{displaymath}
  h_r(\theta) = \sum_{k\in\mathbb{Z}}\frac{2}{r+1}\cos(k\theta)\;,\ 
  h_i(\theta) = -\sum_{k\ne0}\frac{2}{k}Z_k\sin(k\theta)\quad
  \Rightarrow\quad
  h_i' = 2-(r+1)h_r\;.
\end{displaymath}
Now we turn the attention to $F(e^{i\theta})$ and to this end define the odd sequence:
\begin{displaymath}
  Y_k = \left\{\begin{array}{lcl}
    (\zeta_k^{r+1,r+1}+\zeta_k^{r,r})Z_k\;, & & -r\le k\le r\;, \\
    \frac{Z_r}{(2r+2)(2r+1)}\;, & & k=r+1\;, \\
    -\frac{Z_r}{(2r+2)(2r+1)}\;, & & k=-r-1\;, \\
    0\;, & & \abs{k}>r+1\;,
  \end{array}\right.
\end{displaymath}
it follows that for $-r\le k\le r$, $k\ne0$:
\begin{equation}\label{eq:stab_f_t}
  t_k 
  = \frac{1+k(\zeta^{r+1,r+1}_k+\zeta^{r,r}_k+\frac{1}{r+1-k})}{k}\left(\frac{2}{k}-\frac{2}{r+1}\right)Z_k 
  = \left(\frac{2}{k}-\frac{2}{r+1}\right)Y_k + \frac{2}{k^2}Z_k\;.
\end{equation}
and
\begin{displaymath}
  t_{-(r+1)} = \frac{2}{(r+1)^2}C^{r+1,r}_{-r-1}C^{r+1,r}_{-r-1} = \frac{2}{(r+1)^2(2r+1)}Z_r = \frac{4}{r+1}Y_{r+1}\;.
\end{displaymath}
It is thus easy to see~\cref{eq:stab_f_t} holds for all $k\ne0$; 
and the real and imaginary parts of $F(e^{i\theta})$ are:
\begin{displaymath}
  f_r(\theta) = \sum_{k\ne0}\left(\frac{2Y_k}{k}+\frac{2Z_k}{k^2}\right)(\cos(k\theta)-1)\;,\ 
  f_i(\theta) = -\sum_{k\in\mathbb{Z}}\frac{2Y_k}{r+1}\sin(k\theta)\ \Rightarrow\ 
  f_r' = (r+1)f_i+h_i\;.
\end{displaymath}
By~\cref{sec:stab_h}, $h_r>0$ for all $0<\theta<2\pi$, thus we may define:
\begin{displaymath}
  m(\theta) = \frac{f_i(\theta)}{h_r(\theta)},
\end{displaymath}
and compute:
\begin{displaymath}
  q(\theta) 
  = h_r(h_rf_r+h_if_i)+f_i^2 
  = h_r^2(m^2+mh_i+f_r)\;.
\end{displaymath}
Let $S(\theta)=m^2+mh_i+f_r$, we can compute its derivative:
\begin{align}
  \label{eq:stab_f_ds}
  S' 
  = 2m m'+m'h_i+m(2-(r+1)h_r)+(r+1)f_i+h_i 
     = (m'+1)(2m+h_i)\;.
\end{align}
As $f_i(0)=f_r(0)=0$, one has $S(0)=0$; the plan is to show $S'(\theta)<0$ for $0<\theta<\pi$, which indicates $S(\theta)<0$ on $(0,\pi]$ and the negativity of $S$ over the other half $\pi\le\theta<2\pi$ can be obtained by symmetry.
Noting that:
\begin{displaymath}
  (r+2+k)(r+1+k)Z_{k+1} = (r+1-k)(r-k)Z_k\;,\quad\forall k\;,
\end{displaymath}
thus we have:
\begin{align*}
  \sum_{k\in\mathbb{Z}}(r+1-k)(r-k)Z_ke^{i(k+1/2)\theta} 
  &= \sum_{k=\mathbb{Z}}(r+1+k)(r+k)Z_ke^{i(k-1/2)\theta}\;.
\end{align*}
Writing for each $n\ge0$:
\begin{displaymath}
  [Z]^c_n = \sum_{k\in\mathbb{Z}}k^nZ_k\cos(k\theta)\;,\quad
  [Z]^s_n = \sum_{k\in\mathbb{Z}}k^nZ_k\sin(k\theta)\;,
\end{displaymath}
then comparing the imaginary parts of the two sides in the previous identity gives:
\begin{displaymath}
  (2r+1)\cos\frac{\theta}{2}[Z]^s_1 = \sin\frac{\theta}{2}\left(r(r+1)[Z]^c_0+[Z]^c_2\right)\;.
\end{displaymath}
Noting $(r+1)h_r=2[Z]^c_0$, which indicates $(r+1)h_r'=-2[Z]^s_1$, $(r+1)h_r''=-2[Z]^c_2$, one has:
\begin{equation}\label{eq:stab_f_ode_hr}
  h_r''-(2r+1)\cot\frac{\theta}{2}h_r'-r(r+1)h_r = 0\;,\quad 0<\theta<2\pi\;.
\end{equation}
Similarly, for all $-r\le k\le r-1$:
\begin{align*}
  Y_{k+1} &= \left(\zeta^{r+1,r+1}_k+\zeta^{r,r}_k+\frac{1}{r+2+k}+\frac{1}{r+1-k}+\frac{1}{r+1+k}+\frac{1}{r-k}\right)Z_{k+1} \\
  &= \frac{(r+1-k)(r-k)}{(r+2+k)(r+1+k)}Y_k + \frac{2r+3+2k}{(r+2+k)(r+1+k)}Z_{k+1}+\frac{2r+1-2k}{(r+2+k)(r+1+k)}Z_k\;,
\end{align*} 
or equivalently:
\begin{equation}\label{eq:stab_f_rec_y}
  (r+2+k)(r+1+k)Y_{k+1} = (r+1-k)(r-k)Y_k + (2r+3+2k)Z_{k+1} + (2r+1-2k)Z_k\;,
\end{equation}
and one can easily verify that it holds for all integers $k$.
Defining the moments:
\begin{displaymath}
  [Y]^c_n = \sum_{k\in\mathbb{Z}}k^nY_k\cos(k\theta)\;,\quad
  [Y]^s_n = \sum_{k\in\mathbb{Z}}k^nY_k\sin(k\theta)\;,
\end{displaymath}
Multiplying~\cref{eq:stab_f_rec_y} by $e^{i(k+1/2)\theta}$, summing over $k\in\mathbb{Z}$, and comparing the real parts, one gets:
\begin{displaymath}
  r(r+1)\sin\frac{\theta}{2}[Y]^s_0+(2r+1)\cos\frac{\theta}{2}[Y]^c_1+\sin\frac{\theta}{2}[Y]^s_2 = (2r+1)\cos\frac{\theta}{2}[Z]^c_0 + 2\sin\frac{\theta}{2}[Z]^s_1\;.
\end{displaymath}
Since $f_i=\frac{2}{r+1}[Y]^s_0$, $f_i'=\frac{2}{r+1}[Y]^c_1$, and $f_i''=-\frac{2}{r+1}[Y]^s_2$, one has:
\begin{equation}\label{eq:stab_f_ode_fi}
  f_i''-(2r+1)\cot\frac{\theta}{2}f_i'-r(r+1)f_i = (2r+1)\cot\frac{\theta}{2}h_r-2h_r'\;.
\end{equation}

Computing $f_i\times\cref{eq:stab_f_ode_hr}-h_r\times\cref{eq:stab_f_ode_fi}$ gives:
\begin{displaymath}
  f_ih_r''-h_rf_i''+(2r+1)\cot\frac{\theta}{2}(f_i'h_r-f_ih_r')-2h_r'h_r+(2r+1)\cot\frac{\theta}{2}h_r^2 = 0
\end{displaymath}
or
\begin{displaymath}
  (h_r^2+f_i'h_r-f_ih_r')' = (2r+1)\cot\frac{\theta}{2}(h_r^2+f_i'h_r-f_ih_r')\;.
\end{displaymath}
Solving this ODE for $h_r^2+f_i'h_r-f_ih_r'$ leads to:
\begin{equation}\label{eq:stab_f_ode_sol}
  h_r^2+f_i'h_r-f_ih_r' = C_2e^{-(2r+1)\int_{\theta}^\pi\cot\frac{\theta'}{2}d\theta'}
  = C_2\sin^{2(2r+1)}\frac{\theta}{2}\;,
\end{equation}
for some constant $C_2$. 
Moving on, \cref{eq:stab_f_ode_sol} can be written as:
\begin{displaymath}
  h_r^2 + h_r^2m' = C_2\sin^{2(2r+1)}\frac{\theta}{2}\,
\end{displaymath}
Because $m(0)=m(\pi)=0$, one has:
\begin{displaymath}
  0 = \int_0^\pi m'(\theta)\,d\theta = \int_0^\pi \left[\frac{C_2\sin^{2(2r+1)}\frac{\theta}{2}}{[h_r(\theta)]^2}-1\right]\,d\theta\;,
\end{displaymath}
therefore one must have $C_2>0$; hence the first term in the expression of $S'$ in~\cref{eq:stab_f_ds} is always positive: $1+m'>0$ for all $0<\theta<2\pi$.

Now we consider the second term in~\cref{eq:stab_f_ds}, i.e., $2m+h_i$.
On the one hand, it is clear $2m(0)+h_i(0)=2m(\pi)+h_i(\pi)=0$.
On the other hand, one computes:
\begin{align*}
  (2m+h_i)' = 2m'+h_i' = 2m'+2-(r+1)h_r = \frac{2C_2\sin^{2(2r+1)}\frac{\theta}{2}}{[h_r(\theta)]^2} - (r+1)h_r(\theta)\;.
\end{align*}
Because $h_r(\theta)>0$ is strictly decreasing on $(0, \pi]$, see~\cref{sec:stab_h}, we see $(2m+h_i)'$ is a strictly increasing function of $\theta$; therefore $\int_0^{\pi}(2m+h_i)'d\theta=0$ indicates $2m(\theta)+h_i(\theta)$ first decreases from $0$ then increases to $0$ over the interval $[0, \pi]$, and we conclude that $2m+h_i<0$ for all $0<\theta<\pi$.
This concludes the proof that $S'(\theta)<0$ and by the discussion below~\cref{eq:stab_f_ds}, \cref{eq:prelim_stab_cond_g} holds for the first case $(l,r,l',r')=(r+1,r,r,r)$.

\medskip

\noindent
{\it Case 2: $(l,r,l',r')=(r,r,r,r-1)$}. 
We shall modify the definitions of the $Z$- and the $Y$-sequences and then proceed similarly as before:
\begin{align*}
  \beta_0 = \frac{2}{r} = \frac{2}{r}Z_0\;;\quad\beta_k = \left(\frac{2}{r}-\frac{2}{k}\right)Z_k\;,\quad \forall k\in\mathbb{Z},\ k\ne0\;, 
\end{align*}
where $Z_k = [C^{r,r}_k]^2$ for $\abs{k}\le r$ and zero otherwise; and:
\begin{displaymath}
  t_k 
  = \left(\frac{2}{k}-\frac{2}{r}\right)Y_k + \frac{2}{k^2}Z_k\;,\quad -r\le k\le r,\ k\ne0\;,
\end{displaymath}
where $Y_k = 2\zeta_k^{r,r}Z_k, \forall k$.
Using these symbols, one has:
\begin{align*}
  &h_r = \frac{2}{r}\sum_{k\in\mathbb{Z}}Z_k\cos(k\theta)\,,\quad
  h_i = -\sum_{k\ne0}\frac{2}{k}Z_k\sin(k\theta)\,,\quad
  h_i' = 2-rh_r\;; \\
  &f_r = \sum_{k\ne0}\left(\frac{2Y_k}{k}+\frac{2Z_k}{k^2}\right)(\cos(k\theta)-1)\,,\quad
  f_i = -\frac{2}{r}\sum_{k\in\mathbb{Z}}Y_k\sin(k\theta)\,,\quad
  f_r' = rf_i+h_i\;.
\end{align*}
Therefore, defining $m=f_i/h_r$ one still has $q=h_r^2S$ with $S=m^2+mh_i+f_r$, which satisfies $S(0)=0$ and the same derivative formula:
\begin{displaymath}
  S'=2m m'+m'h_i+m(2-rh_r)+rf_i+h_i=(m'+1)(2m+h_i)\;.
\end{displaymath}
One can verify $(r+1+k)^2Z_{k+1}=(r-k)^2Z_k$ and $(r+1+k)^2Y_{k+1}=(r-k)^2Y_k+2(r+1+k)Z_{k+1}+2(r-k)Z_k$; multiplying them by $e^{i(k+1/2)\theta}$, summing over all $k$, and comparing the imaginary parts and real parts of the two resulting equations, respectively, one gets:
\begin{align*}
  &-r^2\sin\frac{\theta}{2}[Z]^c_0+2r\cos\frac{\theta}{2}[Z]^s_1-\sin\frac{\theta}{2}[Z]^c_2 = 0\;, \\ 
  &r^2\sin\frac{\theta}{2}[Y]^s_0+2r\cos\frac{\theta}{2}([Y]^c_1-[Z]^c_0)+\sin\frac{\theta}{2}([Y]^s_2-2[Z]^s_1) = 0\;, 
\end{align*}
which gives:
\begin{displaymath}
  h_r'' - 2r\cot\frac{\theta}{2}h_r' - r^2h_r  = 0\;,\quad
  f_i''-2r\cot\frac{\theta}{2}f_i'-r^2f_i = 2r\cot\frac{\theta}{2}h_r-2h_r'\;.
\end{displaymath}
Therefore:
\begin{displaymath}
  \frac{(h_r^2+f_i'h_r-f_ih_r')'}{h_r^2+f_i'h_r-f_ih_r'}=2r\cot\frac{\theta}{2}\quad\Rightarrow\quad
  h_r^2(1+m') = h_r^2+f_i'h_r-f_ih_r' = C_2\sin^{4r}\frac{\theta}{2}
\end{displaymath}
for some $C_2>0$, which indicates $1+m'>0$ for $0<\theta<2\pi$.
As $2m(0)+h_i(0)=2m(\pi)+h_i(\pi)=0$ and $(2m+h_i)'=2m'+2-rh_r=\frac{2C_2\sin^{4p}\frac{\theta}{2}}{h_r^2}-rh_r$ is an increasing function on $(0,\pi]$, one gets $2m+h_i<0$ for $0<\theta\le\pi$, which shows $S'<0$ and completes the proof.

\medskip

\noindent
{\it Case 3: $(l,r,l',r')=(r+1,r,r+1,r)$}. 
We define $Z_k=\left[C^{r+1,r+1}_k\right]^2,\ \abs{k}\le r+1$ and $Z_k=0,\ \abs{k}>r+1$, and $Y_k=2\zeta_k^{r+1,r+1}Z_k,\ \forall k$, then:
\begin{align*}
  &\beta_0 = \frac{4}{r+1}\;;\quad
  \beta_k = \left(\frac{4}{r+1}-\frac{2k}{(r+1)^2}-\frac{2}{k}\right)Z_k\;,\quad\forall k\ne0\;; \\
  &t_k 
  = \left(\frac{2k}{(r+1)^2}+\frac{2}{k}-\frac{4}{r+1}\right)Y_k + \left(\frac{2}{k^2}-\frac{2}{(r+1)^2}\right)Z_k\;,\quad \abs{k}\le r+1,\ k\ne0\;. 
\end{align*}
Defining the moments $[Z]^c_n$, $[Z]^s_n$, $[Y]^c_n$, and $[Y]^s_n$ as before, one has:
\begin{align*}
  &h_r = \frac{4[Z]^c_0}{r+1}\,,\ 
  h_i = -\frac{2[Z]^s_1}{(r+1)^2}-\sum_{k\ne0}\frac{2}{k}Z_k\sin(k\theta)\ \Rightarrow\ 
  h_i' = 2-\frac{r+1}{2}h_r-\frac{2[Z]^c_2}{(r+1)^2}\,, \\
  &f_r = \sum_{k\ne0}\left[\left(\frac{2k}{(r+1)^2}+\frac{2}{k}\right)Y_k+\left(\frac{2}{k^2}-\frac{2}{(r+1)^2}\right)Z_k\right](\cos(k\theta)-1)\,, \\
  &f_i = -\frac{4}{r+1}[Y]^s_0\ \Rightarrow\ f_r' = \frac{r+1}{2}f_i+h_i+\frac{4[Z]^s_1-2[Y]^s_2}{(r+1)^2}\;.
\end{align*}
In this case, defining $m$ and $S$ as before one obtains:
\begin{equation}\label{eq:stab_f_c3_dS}
  S' = 2m m'+m'h_i + mh_i' + f_r' = (m'+1)(2m+h_i) + m\left(-\frac{2[Z]^c_2}{(r+1)^2}\right) +\frac{4[Z]_1^s-2[Y]^s_2}{(r+1)^2}
\end{equation}

Noticing that $\{Z_k\}$ and $\{Y_k\}$ are precisely those defined in Case 2 when the stencil is given by $(l_0,r_0,l_0',r_0')=(r+1,r+1,r+1,r)$,
we shall denote all symbols (except $Z$ and $Y$) associated with this stencil by the subscript $_0$.
Following the calculation in Case 2 one has:
\begin{align*}
  &h_{0r} = \frac{2}{r_0}[Z]^c_0 = \frac{2}{r+1}[Z]^c_0 = \frac{1}{2}h_r\;, \\
  &h_{0i} = -\sum_{k\ne0}\frac{2}{k}Z_k\sin(k\theta) = h_i+\frac{2[Z]^s_1}{(r+1)^2} = h_i - \frac{1}{r+1}h_{0r}'\;, \\
  &f_{0i} = -\frac{2}{r_0}[Y]^s_0 = -\frac{2}{r+1}[Y]^s_0 = \frac{1}{2}f_i\;.
\end{align*}
Therefore $m_0 = \frac{f_{0i}}{h_{0r}} = \frac{f_i}{h_r} = m$.
Hence moving on from~\cref{eq:stab_f_c3_dS} there is:
\begin{align*}
  S' 
  &= (m_0'+1)\left(2m_0+h_{0i}+\frac{1}{r+1}h_{0r}'\right) + \frac{1}{r+1}m_0h_{0r}'' - \frac{2}{r+1}h_{0r}' - \frac{1}{r+1}f_{0i}'' \\
  &= S_0'+\frac{(m_0'+1)h_{0r}'}{r+1}-\frac{(h_{0r}^2+f_{0i}'h_{0r}-f_{0i}h_{0r}')'}{(r+1)h_{0r}} 
  = S_0'+\frac{(m_0'+1)h_{0r}'}{r+1}-2\cot\frac{\theta}{2}(1+m_0') < 0\;,
\end{align*}
for all $0<\theta<\pi$, where we used $1+m_0>0$, $h_{0r}'<0$, and $\cot\frac{\theta}{2}>0$ on $(0,\pi)$. 

\medskip

\noindent
{\it Case 4: $(l,r,l',r')=(r+1,r,r,r-1)$}. The strategy is similar to the one in the previous case. 
Using the same $\{Z_k\}$ and $\{Y_k\}$ as in Case 1, 
\begin{align*}
  &\beta_0 = 2(\zeta_0^{r+1,r}+\zeta_0^{r,r-1}) = \frac{2(2r+1)}{r(r+1)}Z_0\;;\quad
  \beta_k 
          = \left(\frac{2(2r+1)}{r(r+1)}-\frac{2}{k}-\frac{2k}{r(r+1)}\right)Z_k\;,\quad \forall k\ne0\;, \\
  &t_k = 
    \left(\frac{2}{k}-\frac{2(2r+1)}{r(r+1)}+\frac{2k}{r(r+1)}\right)Y_k + \left(\frac{2}{k^2}-\frac{2}{r(r+1)}\right)Z_k\;,\quad -r-1\le k\le r,\ k\ne0\;. 
\end{align*}
Designating all symbols (except $Y$ and $Z$) associted with the stencil $(l_0,r_0,l_0',r_0')=(r+1,r,r,r)$ by the subscript $_0$, we have the following representations and relations:
\begin{align*}
  &h_r = \frac{2(2r+1)}{r(r+1)}[Z]^c_0 = \frac{2r+1}{r}h_{0r}\;,\ h_i = -\frac{2[Z]^s_1}{r(r+1)}-\sum_{k\ne0}\frac{2}{k}Z_k\sin(k\theta) = h_{0i}+\frac{1}{r}h_{0r}'\;, \\
  &f_r = \sum_{k\ne0}\left[\left(\frac{2}{k}+\frac{2k}{r(r+1)}\right)Y_k+\left(\frac{2}{k^2}-\frac{2}{r(r+1)}\right)Z_k\right](\cos(k\theta)-1)\;,\\
  &f_i = -\frac{2(2r+1)}{r(r+1)}[Y]^s_0 = \frac{2r+1}{r}f_{0i}\;;
\end{align*}
furthremore:
\begin{displaymath}
  h_i' = 2 - \frac{r(r+1)}{2r+1}h_r + \frac{1}{r}h_{0r}''\;,\quad
  f_r' = h_i + \frac{r(r+1)}{2r+1}f_i - \frac{1}{r}f_{0i}''-\frac{2}{r}h_{0r}'\;.
\end{displaymath}
One has $m=\frac{f_i}{h_r} = \frac{f_{0i}}{h_{0r}} = m_0$ and:
\begin{align*}
  S' &= 2m m' + m'h_i + m\left(2-\frac{r(r+1)}{2r+1}h_r+\frac{1}{r}h_{0r}''\right)+h_i+\frac{r(r+1)}{2r+1}f_i-\frac{1}{r}f_{0i}''-\frac{2}{r}h_{0r}' \\
  &= (m_0'+1)(2m_0+h_{0i}) + \frac{(m_0'+1)h_{0r}'}{r} - \frac{(h_{0r}^2+f_{0i}'h_{0r}-f_{0i}h_{0r}')'}{rh_{0r}} \\
  &= S_0' + \frac{(m_0'+1)h_{0r}'}{r} - 2h_{0r}(1+m_0')\cot\frac{\theta}{2} < 0\;,\quad 0<\theta<\pi\;. 
\end{align*}
Here again we used the fact that $S_0'<0$, $m_0'+1>0$, $h_{0r}'<0$, as already proved in Case 1.

\section{$L^2$ stability of central schemes}
\label{sec:l2}
Lastly, we briefly discuss the stability of central HV schemes, where $l=r$ and $l'=r'$.
In this case the $\alpha-$ and $\beta-$ coefficients are:
\begin{align}
  \label{eq:l2_alpha}
  &\alpha_\nu = -\alpha_{-1-\nu}^c\;,\quad -r\le \nu < 0\;;\quad 
  \alpha_\nu = \sum_{k=\nu+1}^r t_k\;,\quad 0\le\nu\le r-1\;; \\
  \label{eq:l2_beta} 
  &\beta_0 = 0\;;\quad
  \beta_\nu = -\frac{2}{\nu}C^{r,r}_\nu C^{r',r'}_\nu = -\beta_{-\nu}\;,\quad -r'\le\nu\le r',\ \nu\ne0\;.
\end{align}
Here $t_k, 1\le k\le r$ are defined by~\cref{eq:prelim_coef_t_pos} and repeated below:
\begin{equation}\label{eq:l2_t}
  \left\{\begin{array}{lcl}
    \textrm{If } r'=r: & & t_k = \frac{2(1+2k\zeta^{r,r}_k)}{k^2}\left[C^{r,r}_k\right]^2\;,\quad 1\le k\le r\;; \\
    \textrm{If } r'=r-1: & & t_k = \left\{\begin{array}{lcl}
      \frac{2(1+k(\zeta^{r,r}_k+\zeta^{r',r'}_k))}{k^2}C^{r,r}_kC^{r',r'}_k\;, & & 1\le k\le r-1\;, \\
      \frac{2}{r^2}C^{r,r}_rC^{r-1,r}_r\;, & & k=r\;.
    \end{array}\right.
  \end{array}\right.
\end{equation}
In this case, $H(e^{i\theta})$ is pure imaginary and $F(e^{i\theta})$ is real:
\begin{displaymath}
  H(e^{i\theta}) = 2i\sum_{k=1}^{r'}\beta_k\sin\theta = ih_i\;,\quad
  F(e^{i\theta}) = 2\sum_{k=1}^rt_k(\cos(k\theta)-1) = f_r\;.
\end{displaymath}
Thus it is easy to verify that $\lambda=\lambda_{1,2}(\theta)\in\mathcal{S}$, see~\cref{eq:prelim_semi_eigset}, are given by:
\begin{align}
  \label{eq:l2_eigval}
  \lambda_{1,2}(\theta) = \frac{1}{2}\left(H \pm \sqrt{H^2+4F}\right) 
  = \frac{i}{2}\left[h_i \pm \sqrt{h_i^2+8\sum_{k=1}^rt_k(1-\cos{k\theta})}\right]\;,
\end{align}
which are all on the imaginary axis as all $t_k$ are positive.

Next we show that~\cref{eq:l2_eigval} indicates the semi-discretized HV scheme are $L^2$-stable.
Consider a uniform grid with $N$ cells and denote the solution vector:
\begin{equation}\label{eq:l2_solvec}
  \bs{u}(t) = \left[\overline{u}_{1/2}, \overline{u}_{3/2}, \cdots, \overline{u}_{N-1/2}, u_0, u_1, \cdots, u_{N-1}\right]^T\;,
\end{equation}
then it solves the linear equation:
\begin{equation}\label{eq:l2_semi}
  \bs{u}' + \frac{1}{h}\bs{M}\bs{u} = \bs{0}\;,\quad\bs{M} = \begin{bmatrix}
    \bs{0} & \bs{S}-\bs{I} \\ G(\bs{S}) & H(\bs{S})
  \end{bmatrix}\;,\quad\textrm{c.f.}~\cref{eq:prelim_semi_mat}. 
\end{equation}
Here $\bs{0}$ and $\bs{I}$ are $N\times N$ zero matrix and identity matrix, respectively; and $\bs{S}$ is cyclic:
\begin{equation}\label{eq:l2_s}
  \bs{S} = \begin{bmatrix}
    0      & 1      & \cdots & 0      & 0      \\
    0      & 0      & \cdots & 0      & 0      \\
    \vdots & \vdots & \ddots & \vdots & \vdots \\
    0      & 0      & \cdots & 0      & 1      \\
    1      & 0      & \cdots & 0      & 0
  \end{bmatrix}\;.
\end{equation}
The matrix $\bs{S}$ is diagonalizable with eigenvalues $s_j=e^{i2j\pi/N}, j=1,\cdots,N$; an unit eigenvector associated with $s_j$ is:
\begin{equation}\label{eq:l2_eigvec}
  \bs{v}_j = \frac{1}{\sqrt{N}}\left[1, s_j, s_j^2, \cdots, s_j^{N-1}\right]^T\;.
\end{equation}
Let $\bs{Q}_j=\begin{bmatrix} 0 & s_j-1 \\ G(s_j) & H(s_j)\end{bmatrix}$, by the discussion below~\cref{eq:prelim_semi_mat} the eigenvalues of $\bs{Q}_j$ are $\lambda_{1,2}(\theta_j)$, $\theta_j=2j\pi/N$ given by~\cref{eq:l2_eigval}, which are distinct.  
Hence $\bs{Q}_j$ is diagonalizable and we may write:
\begin{displaymath}
  \bs{Q}_j = \bs{W}_j\begin{bmatrix} \lambda_1(\theta_j) & 0 \\ 0 & \lambda_2(\theta_j) \end{bmatrix}\bs{W}_j^{-1}\;.
\end{displaymath}
By direct calculation, specially using the fact that $P(\bs{S})\bs{v}_j = P(s_j)\bs{v}_j$ for all Laurent polynomial $P(\cdot)$, one can show that:
\begin{displaymath}
  \bs{Q}(\bs{W}_j\otimes\bs{v}_j) = (\bs{W}_j\otimes\bs{v}_j)\begin{bmatrix} \lambda_1(\theta_j) & 0 \\ 0 & \lambda_2(\theta_j) \end{bmatrix}\;,
\end{displaymath}
therefore the two columns of $\bs{W}_j\otimes\bs{v}_j$ are eigenvectors of $\bs{M}$ associated with the eigenvalues $\lambda_{1,2}(\theta_j)$.
It is trivial to verify that the $2N$ vectors composed of columns of $\bs{W}_j\otimes\bs{v}_j, j=1,2,\cdots,N$ are linearly independent, therefore $\bs{M}$ is diagonalizable:
\begin{equation}\label{eq:l2_m_diag}
  \bs{M} = \bs{U}\bs{D}\bs{U}^{-1},\ 
  \bs{U} = \left[\bs{W}_1\otimes\bs{v}_1\ \cdots\ \bs{W}_N\otimes\bs{v}_N\right],\ 
  \bs{D} = \begin{bmatrix}
    \lambda_1(\theta_1) \\
    & \lambda_2(\theta_1) \\
    & & \ddots \\
    & & & \lambda_2(\theta_N)
  \end{bmatrix}\;.
\end{equation}
By~\cref{eq:l2_eigval}, it is clear that $\bs{D}+\overline{\bs{D}} = \bs{0}$.
To prove the $L^2$-stability of the semi-discretized solution $\bs{u}(t)$, we define $\bs{w}=\bs{U}^{-1}\bs{u}$, which satisfies:
\begin{displaymath}
  \bs{w}' = \bs{U}^{-1}\bs{u}' = -\frac{1}{h}\bs{U}^{-1}\bs{M}\bs{u} = -\frac{1}{h}\bs{D}\bs{w}\;.
\end{displaymath}
Therefore:
\begin{displaymath}
  \frac{d}{dt}\nrm{\bs{w}}^2 = \overline{\bs{w}}^T\bs{w}' + \overline{\bs{w}'}^T\bs{w}
  = \frac{1}{h}\overline{\bs{w}}^T(-\bs{D}\bs{w}) + \frac{1}{h}(-\overline{\bs{D}}\overline{\bs{w}})^T\bs{w}
  = 0\;.
\end{displaymath}
It follows that for all $t\ge0$:
\begin{displaymath}
  \nrm{\bs{u}(t)} = \nrm{\bs{U}\bs{w}(t)} \le \nrm{\bs{U}}\nrm{\bs{w}(t)} = \nrm{\bs{U}}\nrm{\bs{w}(0)} \le \nrm{\bs{U}}\nrm{\bs{U}^{-1}}\nrm{\bs{u}(0)}\;,
\end{displaymath}
where $\nrm{\cdot}$ denotes the $2$-norm;
and it remains to bound the condition number of $\bs{U}$.
\begin{lemma}\label{lm:l2_cond}
  One can construct $\bs{U}$ so that its condition number has the following bound:
  \begin{equation}\label{eq:l2_cond_est}
    \kappa(\bs{U}) \eqdef \nrm{\bs{U}}\nrm{\bs{U}^{-1}} \le \sqrt{\frac{\max(1,\sum_{k=1}^rk^2t_k)}{\min(1,t_1)}}\;.
  \end{equation}
\end{lemma}
\begin{proof}
  First we seek and estimate the eigenvectors of $\bs{Q}(\theta)\eqdef\begin{bmatrix} 0 & e^{i\theta}-1 \\ G(e^{i\theta}) & H(e^{i\theta})\end{bmatrix}$, denoted $\bs{W}(\theta) = [\bs{W}_1(\theta)\ \bs{W}_2(\theta)]$.
  To this end, defining $\phi=\theta/2$:
  \begin{align*}
    e^{i\phi}G(e^{i\theta}) 
    = 2i\sum_{k=1}^rt_k\sum_{\nu=0}^{k-1}\sin(2\nu+1)\phi = 2i\sum_{k=1}^rt_k\frac{\sin^2{k\phi}}{\sin\phi}\;.
  \end{align*}
  Since $\abs{\sin{k\phi}/\sin\phi}\le k$ for all $0<\phi<\pi$, we have:
  \begin{equation}\label{eq:l2_sing}
    \tilde{g}\eqdef\frac{e^{i\phi}G}{2i\sin\phi} = \sum_{k=1}^rt_k\frac{\sin^2{k\phi}}{\sin^2\phi}\quad\Rightarrow\quad
    0< t_1 \le \tilde{g} \le \sum_{k=1}^rk^2t_k\;.
  \end{equation}
  Defining a unitary matrix $\bs{A}(\theta) = \begin{bmatrix} e^{i\theta/4} & 0 \\ 0 & e^{-i\theta/4}\sqrt{\tilde{g}} \end{bmatrix}$, one computes:
  \begin{displaymath}
    \bs{P}(\theta) \eqdef \bs{A}^{-1}\bs{Q}\bs{A} = \begin{bmatrix}
      0 & (e^{i\phi}-e^{-i\phi})\sqrt{\tilde{g}} \\ e^{i\phi}G/\sqrt{\tilde{g}} & H
    \end{bmatrix} = i\begin{bmatrix}
      0 & 2\sin\phi\,\sqrt{\tilde{g}} \\ 2\sin\phi\,\sqrt{\tilde{g}} & h_i 
    \end{bmatrix}\;.
  \end{displaymath}
  Because $\bs{P}(\theta)/i$ is real symmetric, it has real eigenvectors $\tilde{\bs{W}}_1(\theta)$ and $\tilde{\bs{W}}_2(\theta)$ such that the matrix $\tilde{\bs{W}}(\theta) = \left[\tilde{\bs{W}}_1\ \tilde{\bs{W}}_2\right]$ is orthogonal.
  As $\bs{P}$ is similar to $\bs{Q}$, we may construct a pair of eigenvectors of the latter as: $\bs{W} = \bs{A}\tilde{\bs{W}}$.

  Picking $\bs{W}_j = \bs{W}(\theta_j)$ as discussed before, we can compute:
  \begin{align*}
    \bs{U} = \begin{bmatrix}
      \bs{W}_1\otimes\bs{v}_1 & \cdots & \bs{W}_N\otimes\bs{v}_N
    \end{bmatrix} = \begin{bmatrix}
      \bs{I}_2\otimes\bs{v}_1 & \cdots & \bs{I}_2\otimes\bs{v}_N
    \end{bmatrix}\begin{bmatrix}
      \bs{W}_1 \\ & \bs{W}_2 \\ & & \ddots \\ & & & \bs{W}_N
    \end{bmatrix}\;,
  \end{align*}
  where $\bs{I}_2$ is the $2\times2$ identity matrix.
  Since $\begin{bmatrix}\bs{I}_2\otimes\bs{v}_1 & \cdots & \bs{I}_2\otimes\bs{v}_N\end{bmatrix}$ is clearly unitary, one has:
  \begin{displaymath}
    \kappa(\bs{U}) = \frac{\textrm{Largest singular values of }\bs{W}_j \textrm{ across all } 1\le j\le N}{\textrm{Smallest singular values of }\bs{W}_j \textrm{ across all } 1\le j\le N}\;.
  \end{displaymath}
  By previous analysis, $\bs{W} = \bs{A}\tilde{\bs{W}}$, where $\tilde{\bs{W}}$ is orthogonal; thus the singular values of $\bs{W}$ identify with those of $\bs{A}$, i.e., $1$ and $\sqrt{\tilde{g}}$.
  Then the estimate~\cref{eq:l2_cond_est} follows from~\cref{eq:l2_sing}.
\end{proof}
Because the bound of $\kappa(\bs{U})$ given by~\ref{eq:l2_cond_est} is indepenent of $N$, we obtain the following $L^2$-stability result of the central HV method.
\begin{theorem}\label{thm:l2_stab}
  There exists a constant $C$ that is determined by the operator $[\mathcal{D}_x]$ in the central HV scheme (particularly, it does not depend on $h$), such that for all $t\ge0$:
  \begin{equation}\label{eq:l2_stab}
    \sum_{j=0}^{N-1}\left[\overline{u}_{\phf{j}}(t)\right]^2 + \sum_{j=0}^{N-1}\left[u(t)\right]^2 \le C
    \left(\sum_{j=0}^{N-1}\left[\overline{u}_{\phf{j}}(0)\right]^2 + \sum_{j=0}^{N-1}\left[u(0)\right]^2\right)\;.
  \end{equation}
\end{theorem}
It should be noted that this energy estimate extends naturally to fully discretized method that combines the central HV scheme in space and midpoint rule in time, as stated below.
\begin{theorem}\label{thm:l2_stab_mp}
  Denoting the numerical solution to~\cref{eq:l2_semi} at $t^n$ by $\bs{u}^n$ for any integer $n$, and using midpoint rule to update the solution from $t^n$ to $t^{n+1}=t^n+\Delta t^n$:
  \begin{equation}\label{eq:l2_mp}
    \frac{\bs{u}^{n+1}-\bs{u}^n}{\Delta t^n} + \frac{1}{h}\bs{M}\frac{\bs{u}^{n+1}+\bs{u}^n}{2} = \bs{0}\;,
  \end{equation}
  then for all $n>0$, there is the estimate:
  \begin{equation}\label{eq:l2_stab_mp}
    \bs{u}^n\cdot\bs{u}^n \le C \bs{u}^0\cdot\bs{u}^0\;,
  \end{equation}
  where $\bs{u}^0$ is the initial data and $C$ is the same constant in~\cref{thm:l2_stab}.
\end{theorem}
\begin{proof}
  Premultiplying~\cref{eq:l2_mp} by $\bs{U}^{-1}$, one obtains:
  \begin{equation}\label{eq:l2_mp_w}
    \frac{\bs{w}^{n+1}-\bs{w}^n}{\Delta t^n} + \frac{1}{h}\bs{D}\frac{\bs{w}^{n+1}+\bs{w}^n}{2} = \bs{0}\;;
  \end{equation}
  then taking its inner product with $\overline{\bs{w}}^{n+1}+\overline{\bs{w}}^n$, there is:
  \begin{align*}
    \overline{\bs{w}}^{n+1}\cdot\bs{w}^{n+1}-\overline{\bs{w}}^n\cdot\bs{w}^n &= - (\overline{\bs{w}}^{n+1}+\overline{\bs{w}}^n)\cdot\frac{\Delta t^n}{2h}\bs{D}(\bs{w}^{n+1}+\bs{w}^n) \\
    &= - \frac{\Delta t^n}{4h}(\overline{\bs{w}}^{n+1}+\overline{\bs{w}}^n)\cdot(\bs{D}+\overline{\bs{D}})(\bs{w}^{n+1}+\bs{w}^n) = 0\;.
  \end{align*}
  Therefore for all $n\ge0$, $\overline{\bs{w}}^n\cdot\bs{w}^n = \overline{\bs{w}}^0\cdot\bs{w}^0$, and~\cref{eq:l2_stab_mp} follows naturally.
\end{proof}

Lastly, we point out that like the finite difference methods, central HV schemes should not be paired with explicit time integrators, especially when the time step size is computed by the usual Courant condition, see for example the theorem below.
\begin{theorem}\label{thm:l2_instab_fe}
  Using the same notation as in~\cref{thm:l2_stab_mp}, let the solution update from $t^n$ to $t^{n+1}$ be given by the forward Euler method:
  \begin{equation}\label{eq:l2_fe}
    \frac{\bs{u}^{n+1}-\bs{u}^n}{\Delta t} + \frac{1}{h}\bs{M}\bs{u}^n = \bs{0}\;,
  \end{equation}
  where $\Delta t$ is the uniform time step size computed by $\Delta t = \alpha_{\cfl}h^{\gamma}$ for some positive constants $\alpha_{\cfl}$ and $\gamma$.
  Let the solution be solved to $t=T$ after $N_T=T/\Delta t$ time steps, then: (1) if $\gamma<2$, there is $\lim_{h\to0}\frac{\nrm{\bs{u}^{N_T}}}{\nrm{\bs{u}^0}}\to\infty$ for almost all $\bs{u}^0$, (2) if $\gamma=2$, there exists a constant $C_T$ growing exponentially in $T$ such that $\nrm{\bs{u}^{N_T}}\le C_T\bs{u}^0$ for all $h>0$, and (3) if $\gamma>2$, $\lim_{h\to0}\frac{\nrm{\bs{u}^{N_T}}}{\nrm{\bs{u}^0}} = 1$.
\end{theorem}
The proof builds on estimating the eigenvalues of $(\bs{I}-\alpha_{\cfl}h^{\gamma-1}\bs{M})^{\frac{T}{\alpha_{\cfl}h^\gamma}}$ as $h\to0$ and is elementary; hence we omit it here.

\section{Numerical verifications}
\label{sec:num}
Finally, we verify the stability/instability results for HV methods with an upwind biased stencil, as well as the $L^2$ stability of central schemes.
First, we plot the set $\mathcal{S}$ given by~\cref{eq:prelim_semi_eigset} for selected HV methods.
According to the discussion in~\cref{sec:prelim}, the semi-discretized HV method is stable if and only if the entire trajectory lies in the right complex plane.
In~\cref{fg:num_stab_lr1}, we plot selected HV methods with $L-R=1$; and~\cref{fg:num_stab_lr2} plots the eigenvalue trajectories for HV methods with $L-R=2$.
Note that as $L$ and $R$ get larger, the trajectory stay closer to the imaginary axis, which is expected as they construct more accurate HV schemes; the enlarged views nearby the imaginary axis also show that $\mathcal{S}$ is in the right complex plane (these plots are not shown), c.f.~\cref{fg:num_instab}.
\begin{figure}\centering
  \begin{subfigure}[b]{.48\textwidth}\centering
    \includegraphics[trim=1in 0 1in 0.2in, clip, width=\textwidth]{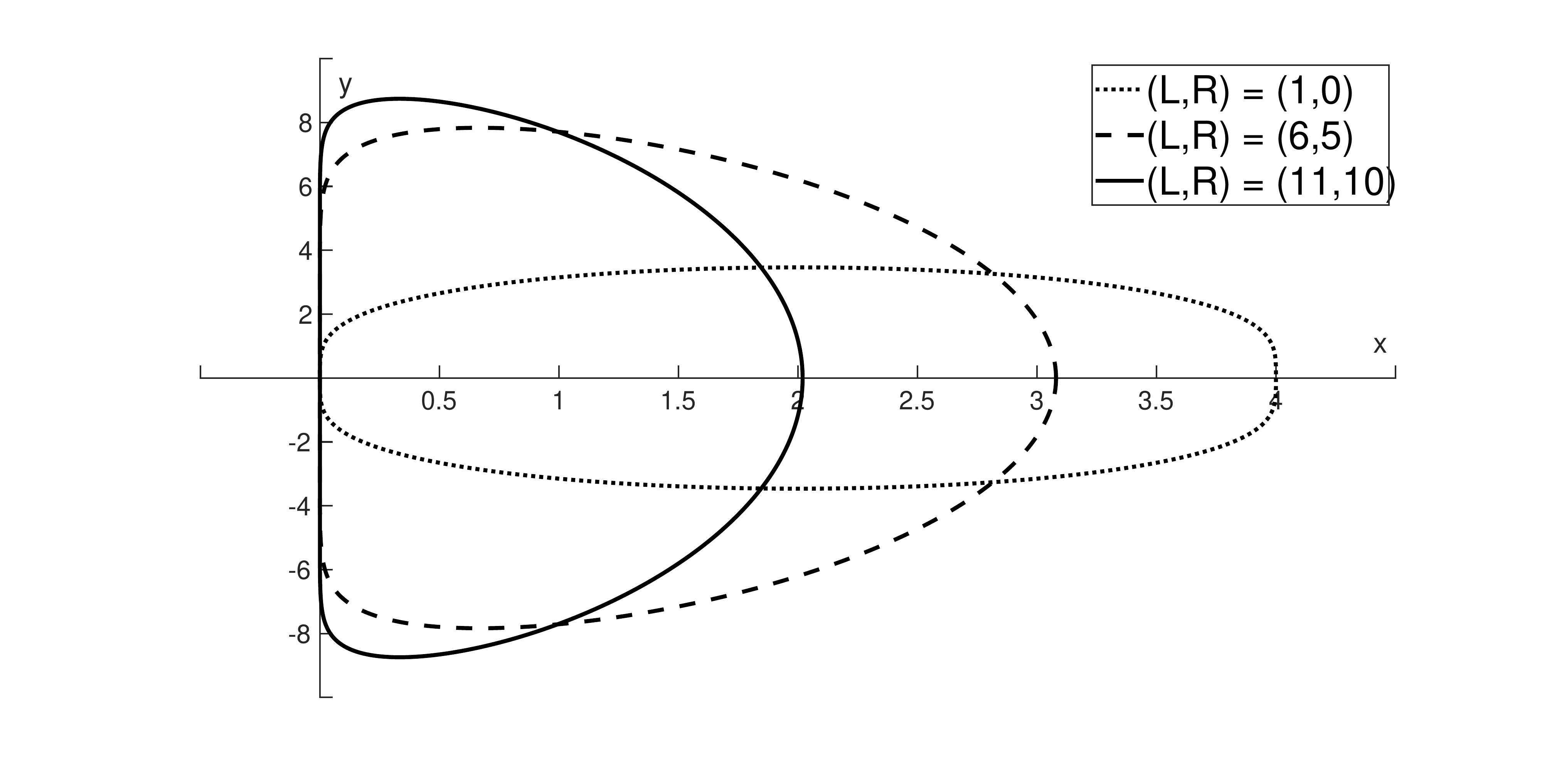}
    \caption{$\mathcal{S}$ in the case $L-R=1$.}
    \label{fg:num_stab_lr1}
  \end{subfigure}
  \begin{subfigure}[b]{.48\textwidth}\centering
    \includegraphics[trim=1in 0 1in 0.2in, clip, width=\textwidth]{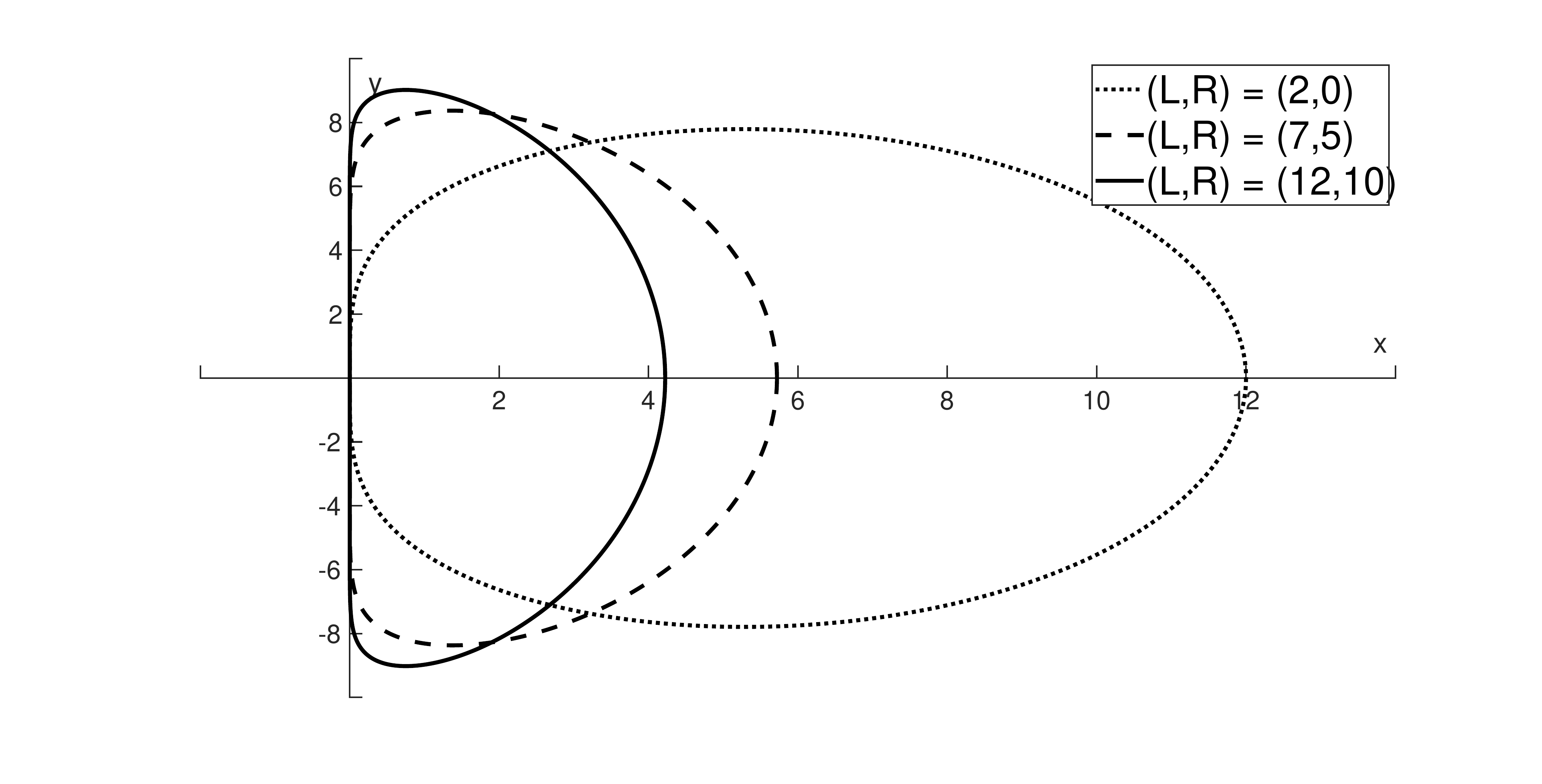}
    \caption{$\mathcal{S}$ in the case $L-R=2$.}
    \label{fg:num_stab_lr2}
  \end{subfigure}
  \vglue -.15in
  \caption{Plotting $\mathcal{S}$ for selected HV methods with stencil $(L,R)$ such that (1) $L-R=1$ (left panel) or (2) $L-R=2$ (right panel). 
    All trajectories are in the right complex plane.}
  \label{fg:num_stab}
\end{figure}

Next, the eigenvalue trajectory plot for selected HV methods with $L-R\ge3$ are plotted -- note that according to~\cref{thm:prelim_sb} all such methods are unstable except when $R=0$ and $L=3$.
The instability for selected schemes with $L-R>3$ is clearly observed in~\cref{fg:num_instab_lr_more}, where the trajectory $\mathcal{S}$ clearly penatrates into the left complex plane.
When $L-R=3$, the instability is not clearly shown in the global vuew (\cref{fg:num_instab_lr3_glob}); but zooming around the imaginary axis shows that the HV scheme when $R\ge1$ is indeed unstable, whereas when $(L,R)=(3,0)$ the trajectory $\mathcal{S}$ always stay in the right complex plane.
\begin{figure}\centering
  \begin{subfigure}[b]{.48\textwidth}\centering
    \includegraphics[trim=1in 0 1in 0.2in, clip, width=\textwidth]{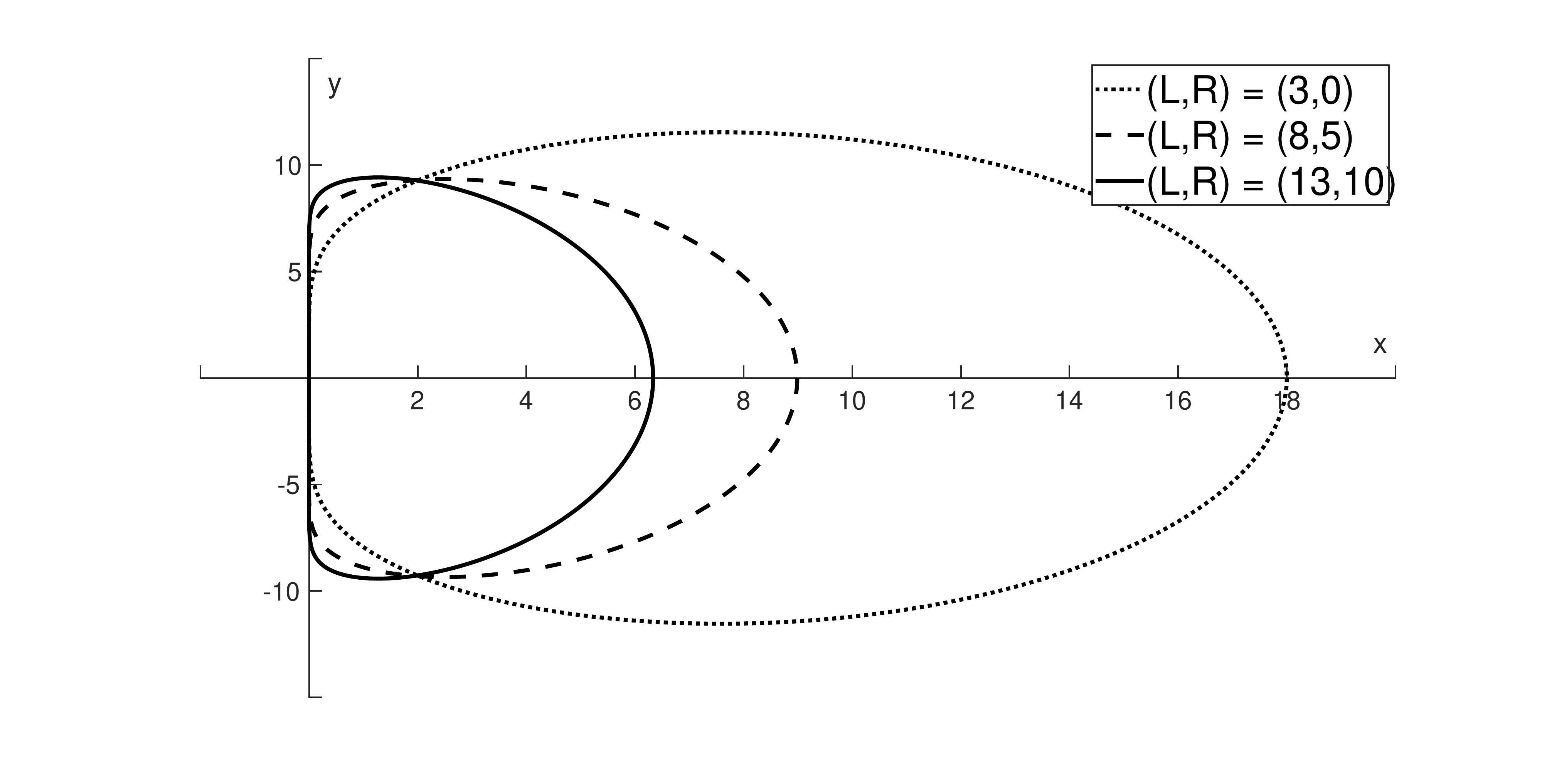}
    \caption{$\mathcal{S}$ in the case $L-R=3$.}
    \label{fg:num_instab_lr3_glob}
  \end{subfigure}
  \begin{subfigure}[b]{.48\textwidth}\centering
    \includegraphics[trim=1in 0 1in 0.2in, clip, width=\textwidth]{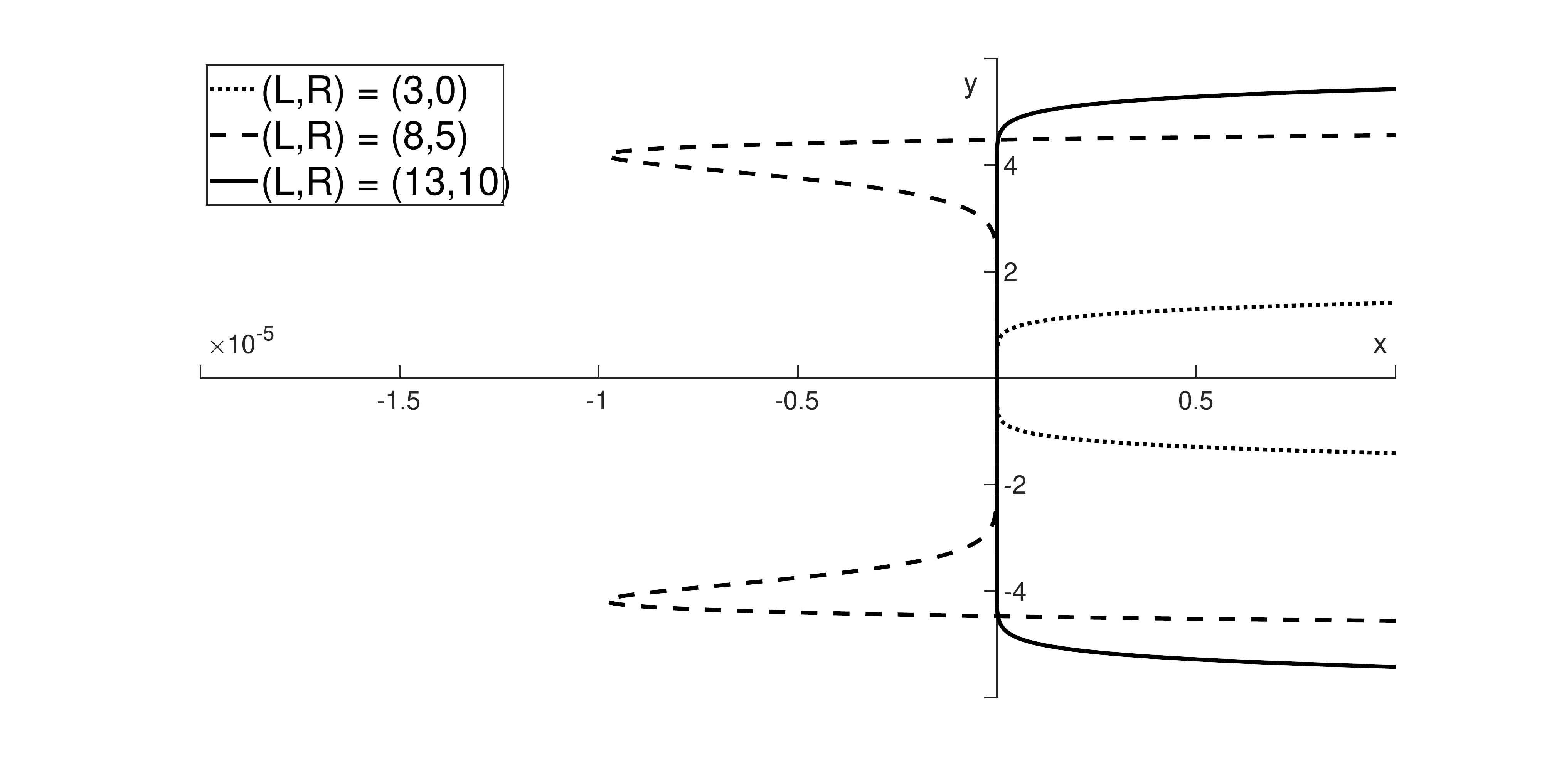}
    \caption{$\mathcal{S}$ in the case $L-R=3$, enlarged view.}
    \label{fg:num_instab_lr3_loc1}
  \end{subfigure} \\
  \begin{subfigure}[b]{.48\textwidth}\centering
    \includegraphics[trim=1in 0 1in 0.2in, clip, width=\textwidth]{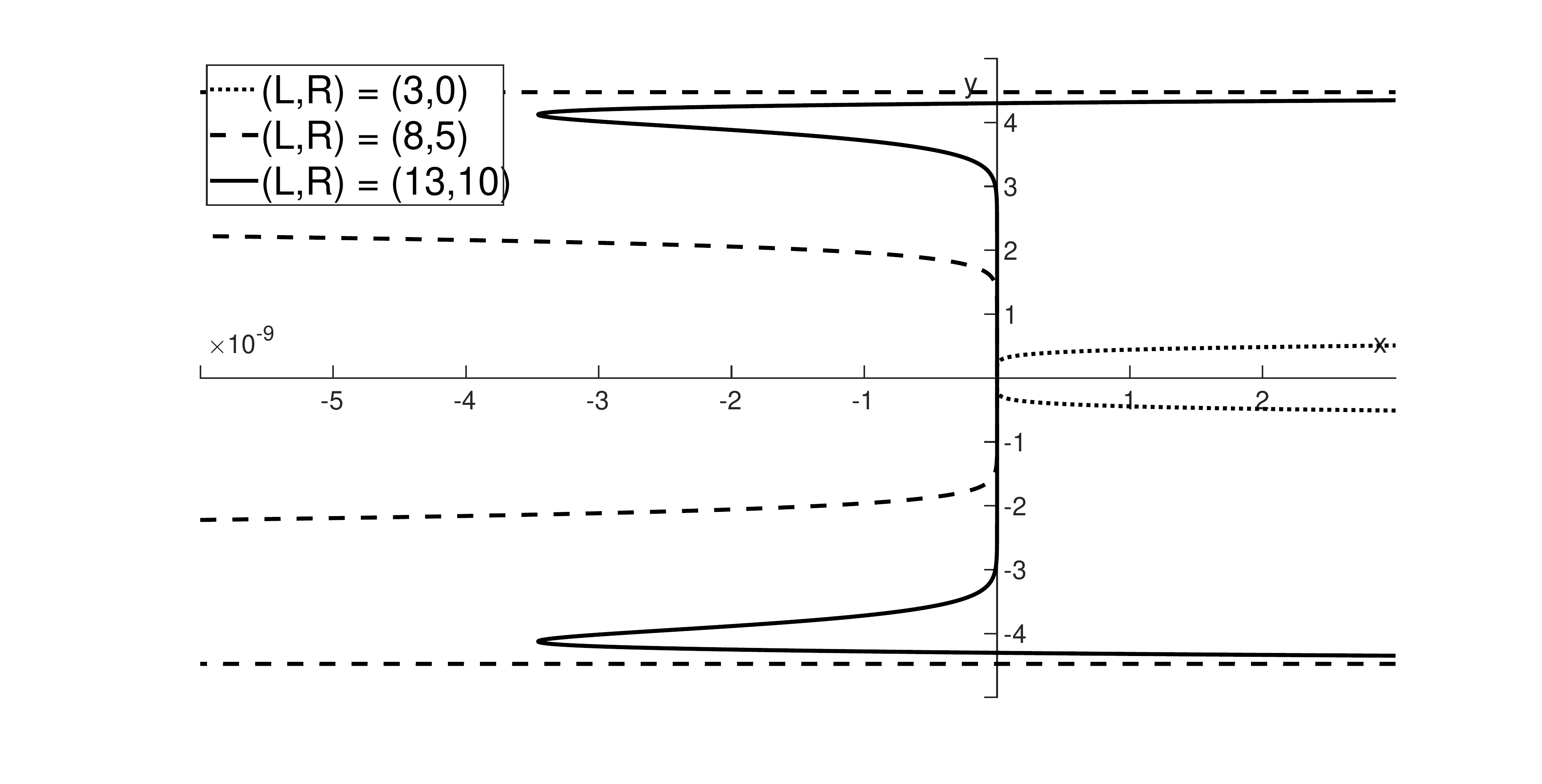}
    \caption{$\mathcal{S}$ in the case $L-R=3$, further enlarged view.}
    \label{fg:num_instab_lr3_loc2}
  \end{subfigure}
  \begin{subfigure}[b]{.48\textwidth}\centering
    \includegraphics[trim=1in 0 1in 0.2in, clip, width=\textwidth]{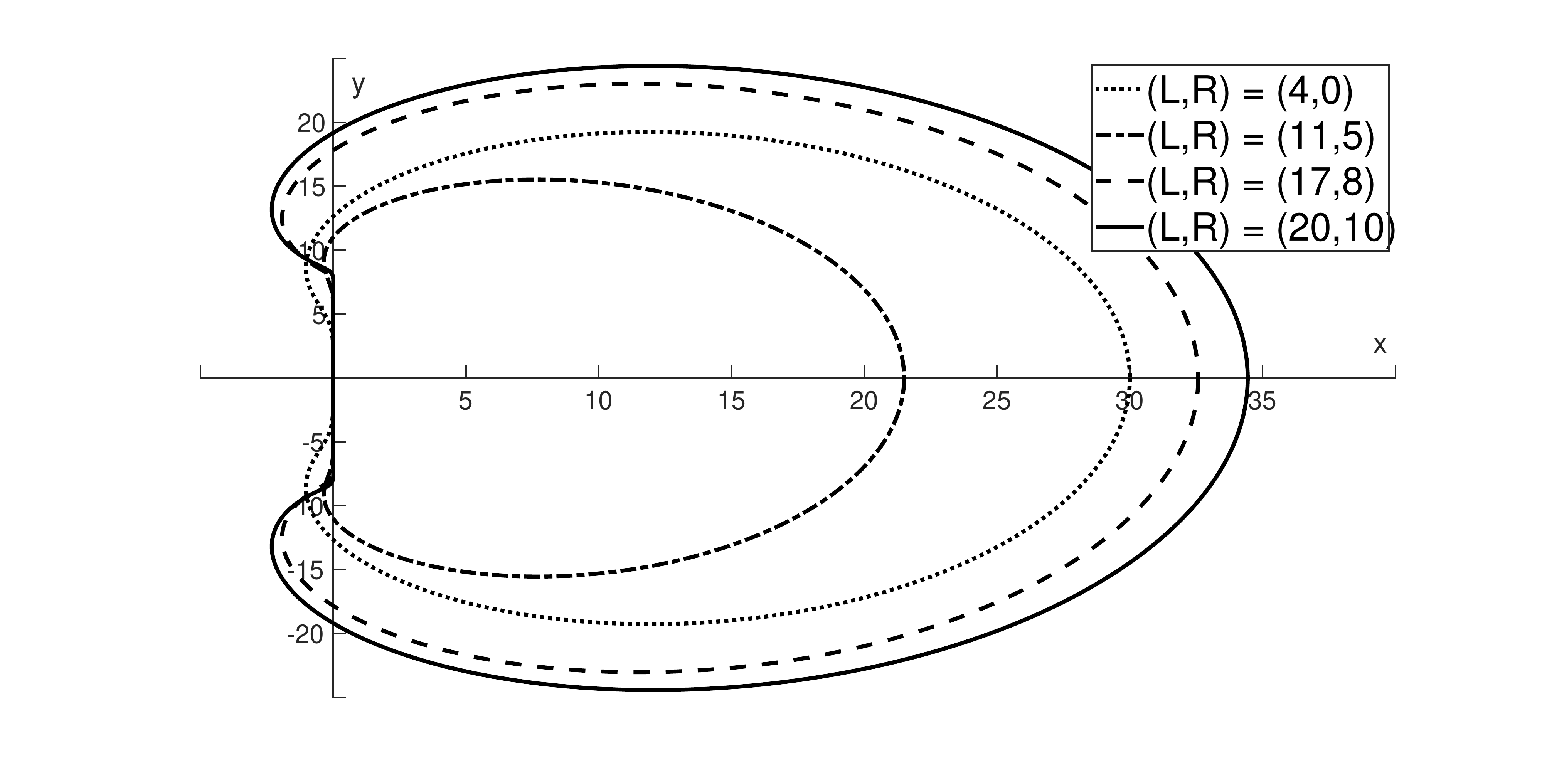}
    \caption{$\mathcal{S}$ in the case $L-R>3$.}
    \label{fg:num_instab_lr_more}
  \end{subfigure}
  \vglue -.15in
  \caption{Plotting $\mathcal{S}$ for selected HV methods with stencil $(L,R)$ such that $L-R\ge3$.
    \cref{fg:num_instab_lr3_glob}--\cref{fg:num_instab_lr3_loc2} plots $\mathcal{S}$ for selected schemes with $L-R=3$ in both a global view and two local views around the imaginary axis; and~\cref{fg:num_instab_lr_more} presents $\mathcal{S}$ of sample schemes with $L-R>3$.}
  \label{fg:num_instab}
\end{figure}

Finally, we verify the stability/instability results in~\cref{thm:l2_stab_mp} and~\cref{thm:l2_instab_fe}.
In both cases, we consider the Cauchy problem with an Gaussian pulse for the initial data:
\begin{equation}\label{eq:num_l2_prob}
  \left\{\begin{array}{lcl}
    u_t + u_x = 0\;, & &  (x,t)\in[0, 1]\times[0, 6]\;, \\
    u(0,t) = u(1,t)\;, & & t\in[0, 10]\;, \\
    u(x,0) = e^{-100\left(x-\frac{1}{2}\right)^2}\;, & & x\in[0, 1]\;.
  \end{array}\right.
\end{equation}
To verify~\cref{thm:l2_stab_mp}, we consider two central HV schemes with stencil given by $(L,R)=(2,2)$ and $(L,R)=(7,7)$, three fixed time step sizes $\Delta t = 0.005$, $0.01$, and $0.02$.
For each combination of $(L,R)$ and $\Delta t$, we use the midpoint method given by~\cref{eq:l2_stab_mp} to compute numerical solutions using two uniform grids with the number of cells given by $N=20$ and $N=160$.
The nodal solutions computed by the HV method with $L=R=2$ are presented in~\cref{fg:num_l2_mp_r2}, and those computed by the HV method with $L=R=7$ are offered in~\cref{fg:num_l2_mp_r7}.
\begin{figure}\centering
  \includegraphics[trim=1in 0 1in .2in, clip, width=.48\textwidth]{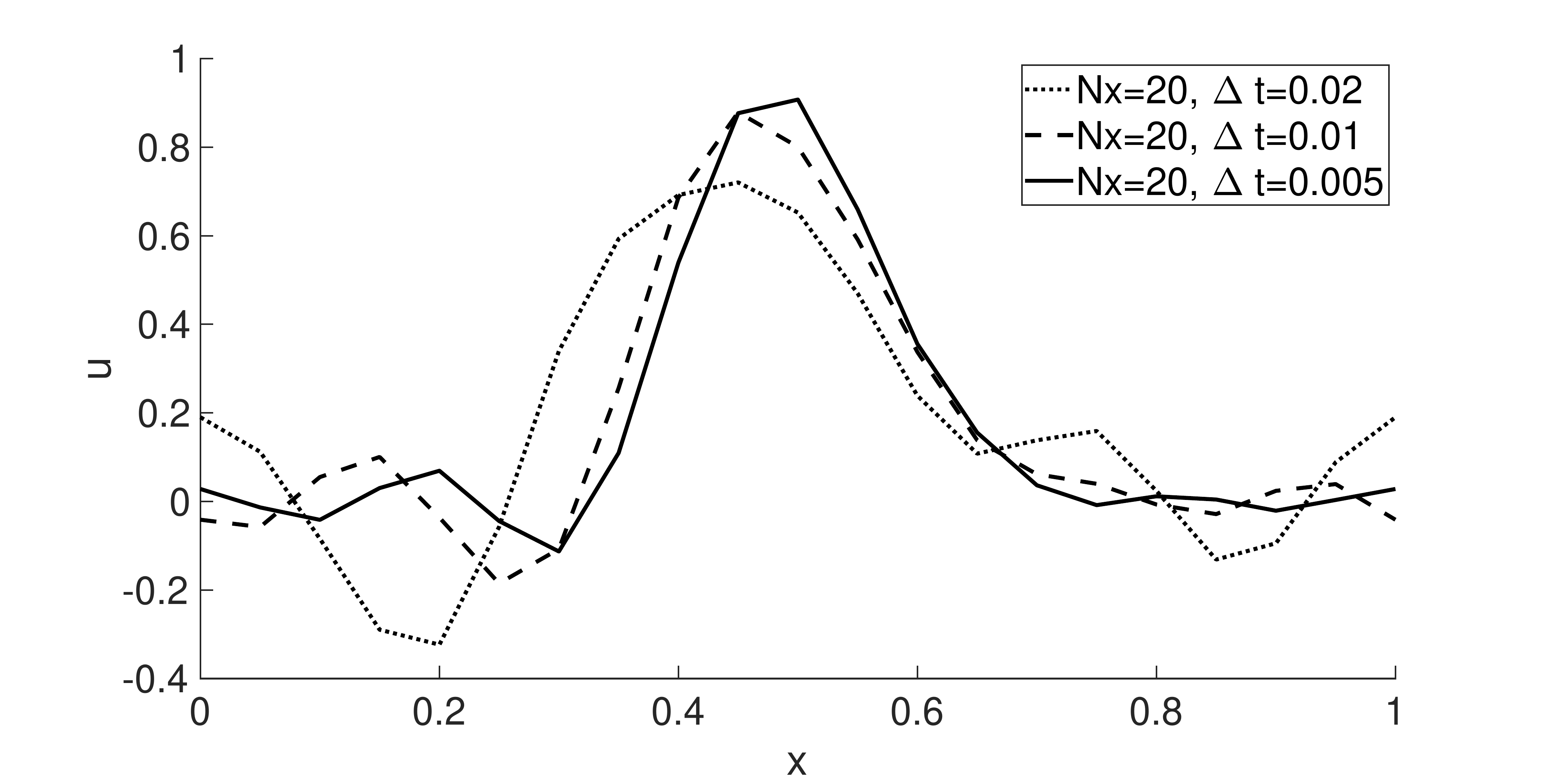} ~~
  \includegraphics[trim=1in 0 1in .2in, clip, width=.48\textwidth]{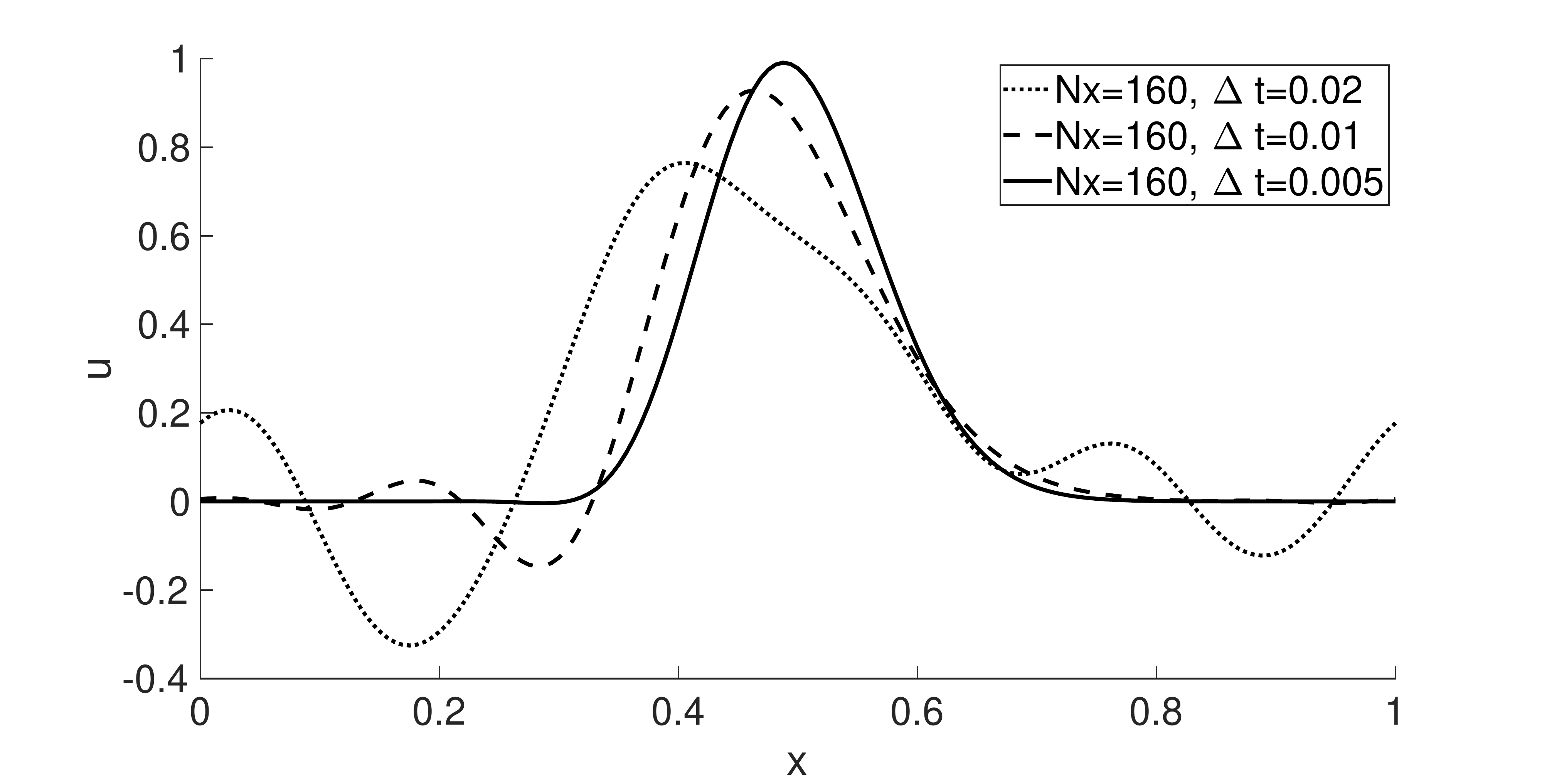}
  \caption{Nodal solutions to~\cref{eq:num_l2_prob} by the central HV scheme with $(L,R)=(2,2)$ in space and midpoint rule~\cref{eq:l2_stab_mp} in time; a uniform grid with $N=20$ cells (left panel) and a uniform grid with $N=160$ cells (right panel) are used.
  Cell-averaged solutions are similar.}
  \label{fg:num_l2_mp_r2}
\end{figure}
\begin{figure}\centering
  \includegraphics[trim=1in 0 1in .2in, clip, width=.48\textwidth]{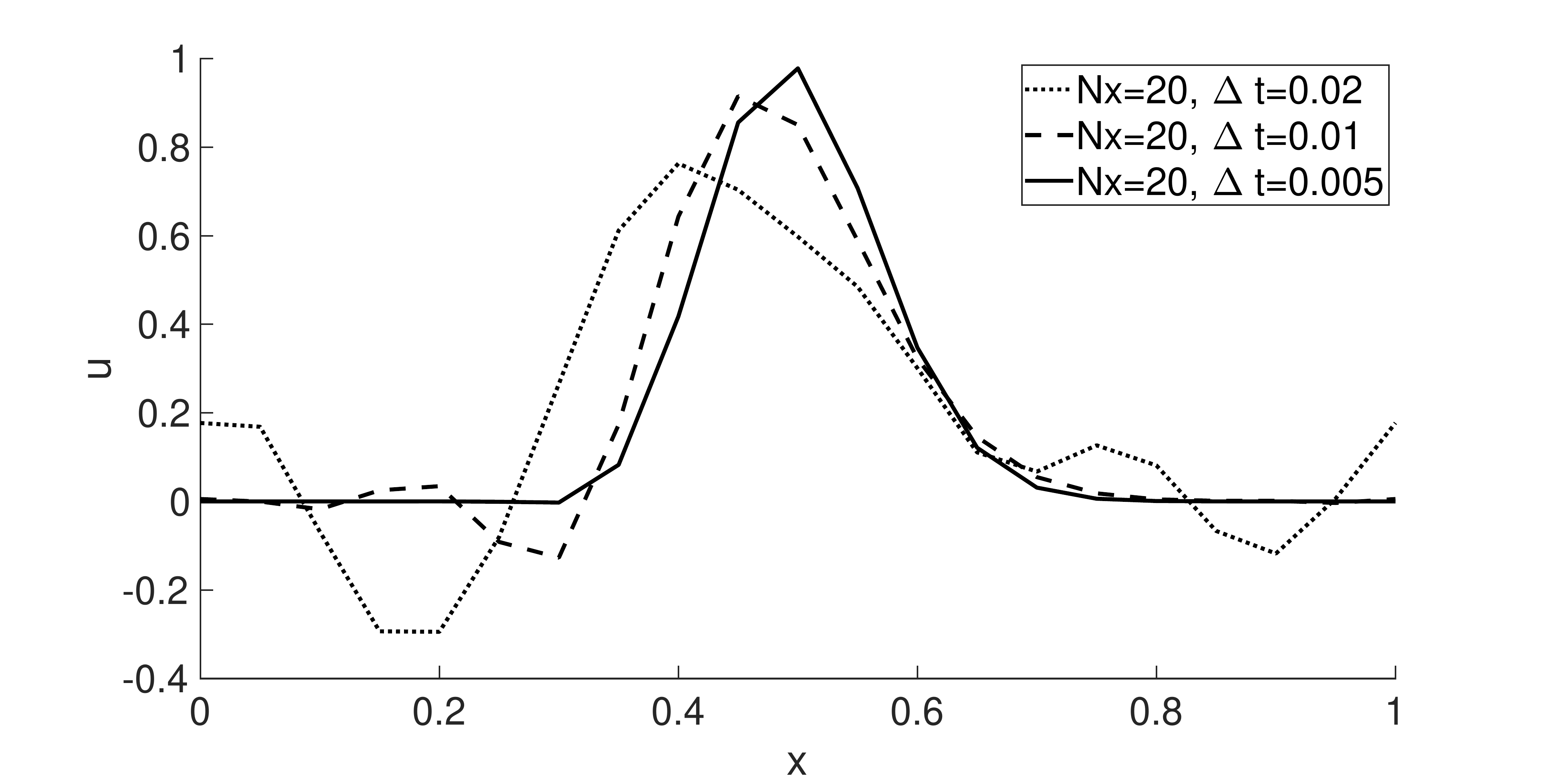} ~~
  \includegraphics[trim=1in 0 1in .2in, clip, width=.48\textwidth]{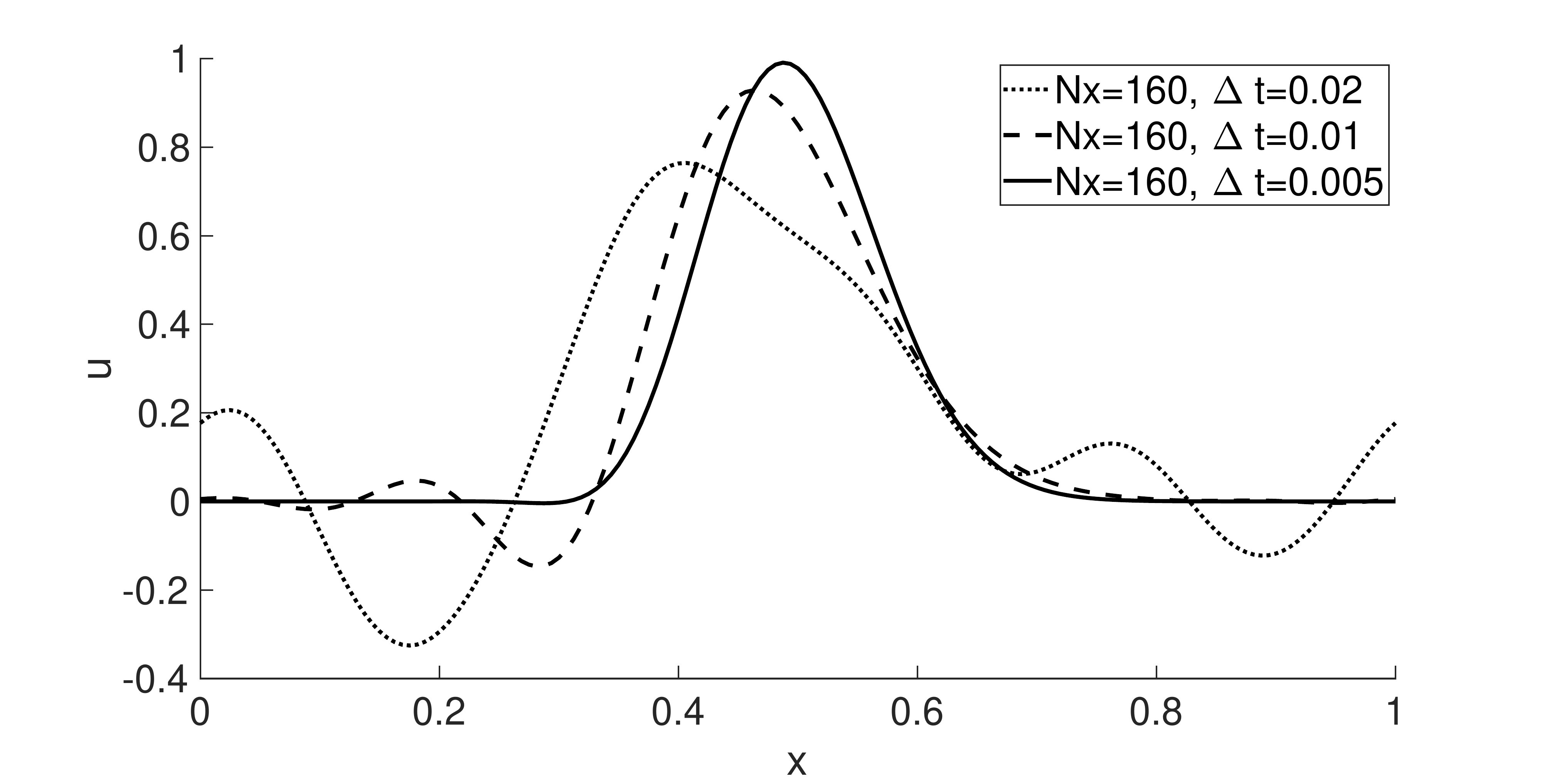}
  \caption{Nodal solutions to~\cref{eq:num_l2_prob} by the central HV scheme with $(L,R)=(7,7)$ in space and midpoint rule~\cref{eq:l2_stab_mp} in time; a uniform grid with $N=20$ cells (left panel) and a uniform grid with $N=160$ cells (right panel) are used.
  Cell-averaged solutions are similar.}
  \label{fg:num_l2_mp_r7}
\end{figure}
From these plots we see that all computations are stable; note that the Courant number $\Delta t/h$ for these computations ranges from $0.1$ to $3.2$.

To verify the predictions by~\cref{thm:l2_instab_fe}, we first consider the HV scheme with $(L,R)=(2,2)$ and compute the time step size using $\Delta t=0.8h^p$ with $p=1,2,3$, where $h$ is the uniform cell size that takes the value $1/20$, $1/40$, $1/80$, and $1/160$.
The relative $L^2$-norms of the numerical solutions at integer $t$ are plotted (logarithmic scale is used if necessary) in~\cref{fg:num_l2_fe_r2_p1}--\cref{fg:num_l2_fe_r2_p3} for the three values of $p$.
We observe that when $p=1$ the numerical solution blows up quickly and the blowing-up speed increases as $h$ decreases.
When $p=2$, all solutions eventually blow up, but the rate at which is independent of the mesh size $h$.
Lastly, the $L^2$-norm of the numerical solutions almost stay constant when $p=3$; and we plot the nodal solutions at $T=10$ in~\cref{fg:num_l2_fe_r2_p3_plot}.
\begin{figure}\centering
  \begin{subfigure}[b]{.48\textwidth}\centering
    \includegraphics[trim=1in 0 1in 0.2in, clip, width=\textwidth]{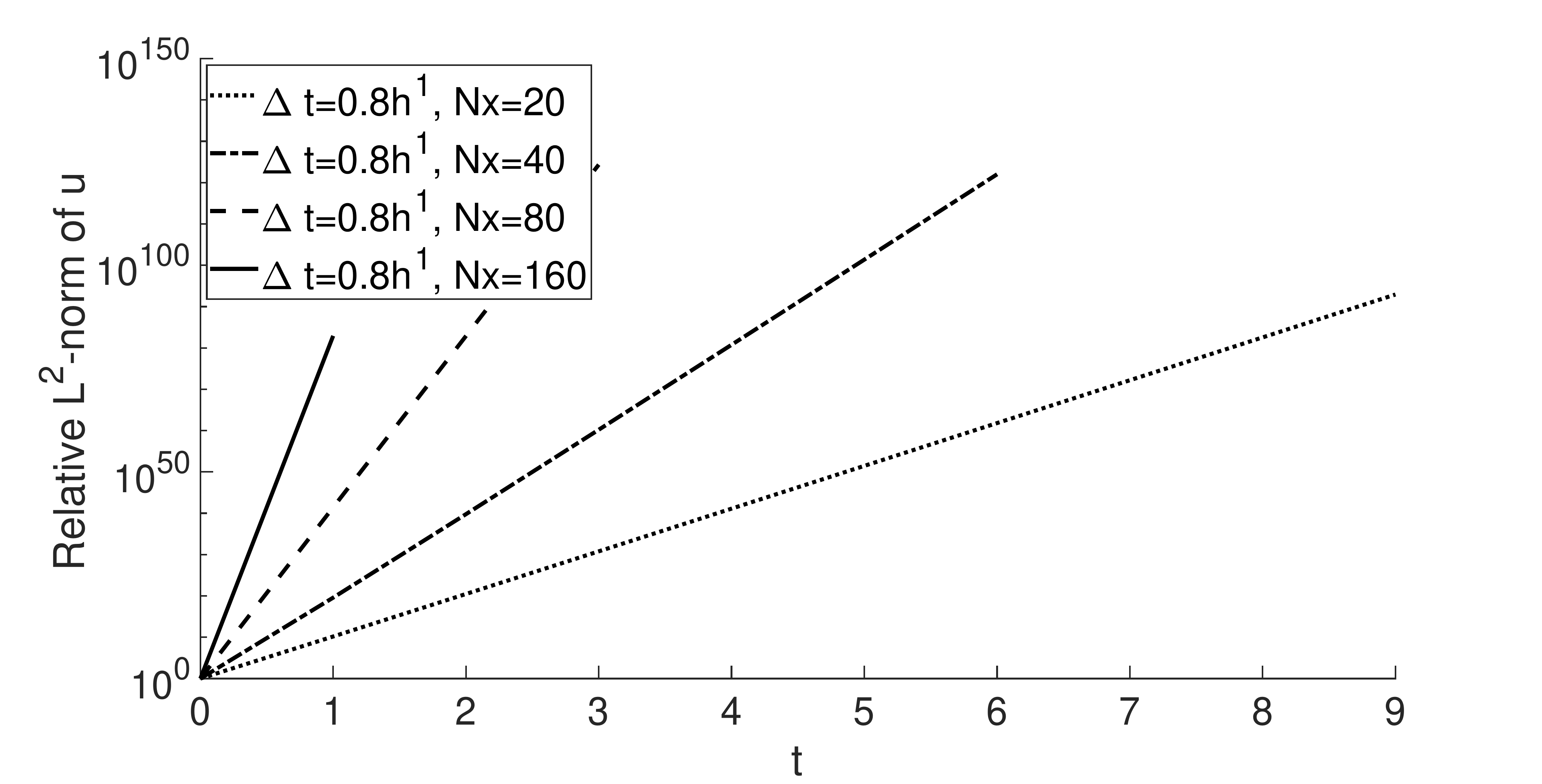}
    \caption{Histories of $\frac{\nrm{\bs{u}^n}}{\nrm{\bs{u}^0}}$ computed by $\Delta t=0.8h$.}
    \label{fg:num_l2_fe_r2_p1}
  \end{subfigure}
  \begin{subfigure}[b]{.48\textwidth}\centering
    \includegraphics[trim=1in 0 1in 0.2in, clip, width=\textwidth]{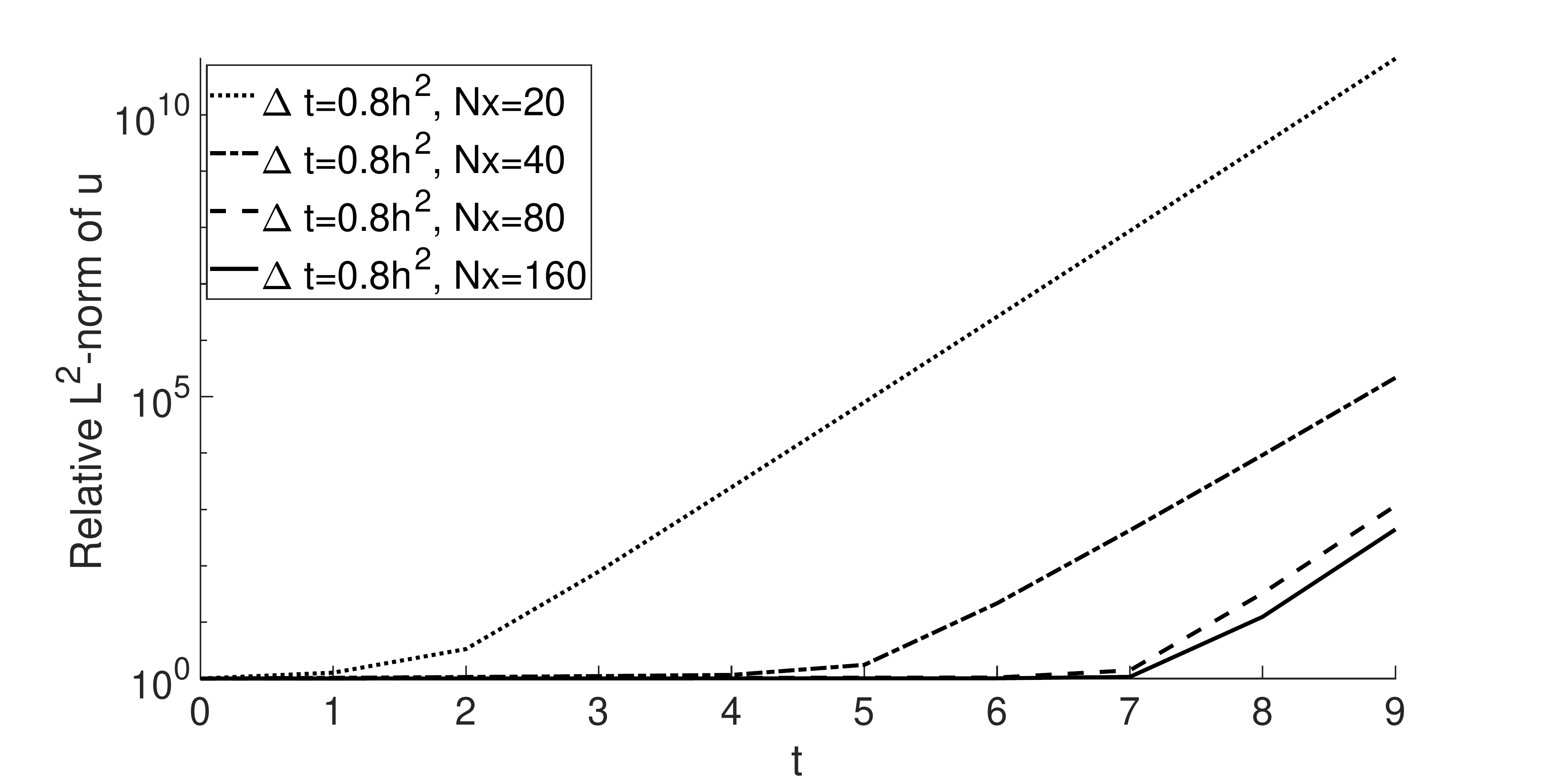}
    \caption{Histories of $\frac{\nrm{\bs{u}^n}}{\nrm{\bs{u}^0}}$ computed by $\Delta t=0.8h^2$.}
    \label{fg:num_l2_fe_r2_p2}
  \end{subfigure} \\
  \begin{subfigure}[b]{.48\textwidth}\centering
    \includegraphics[trim=1in 0 1in 0.2in, clip, width=\textwidth]{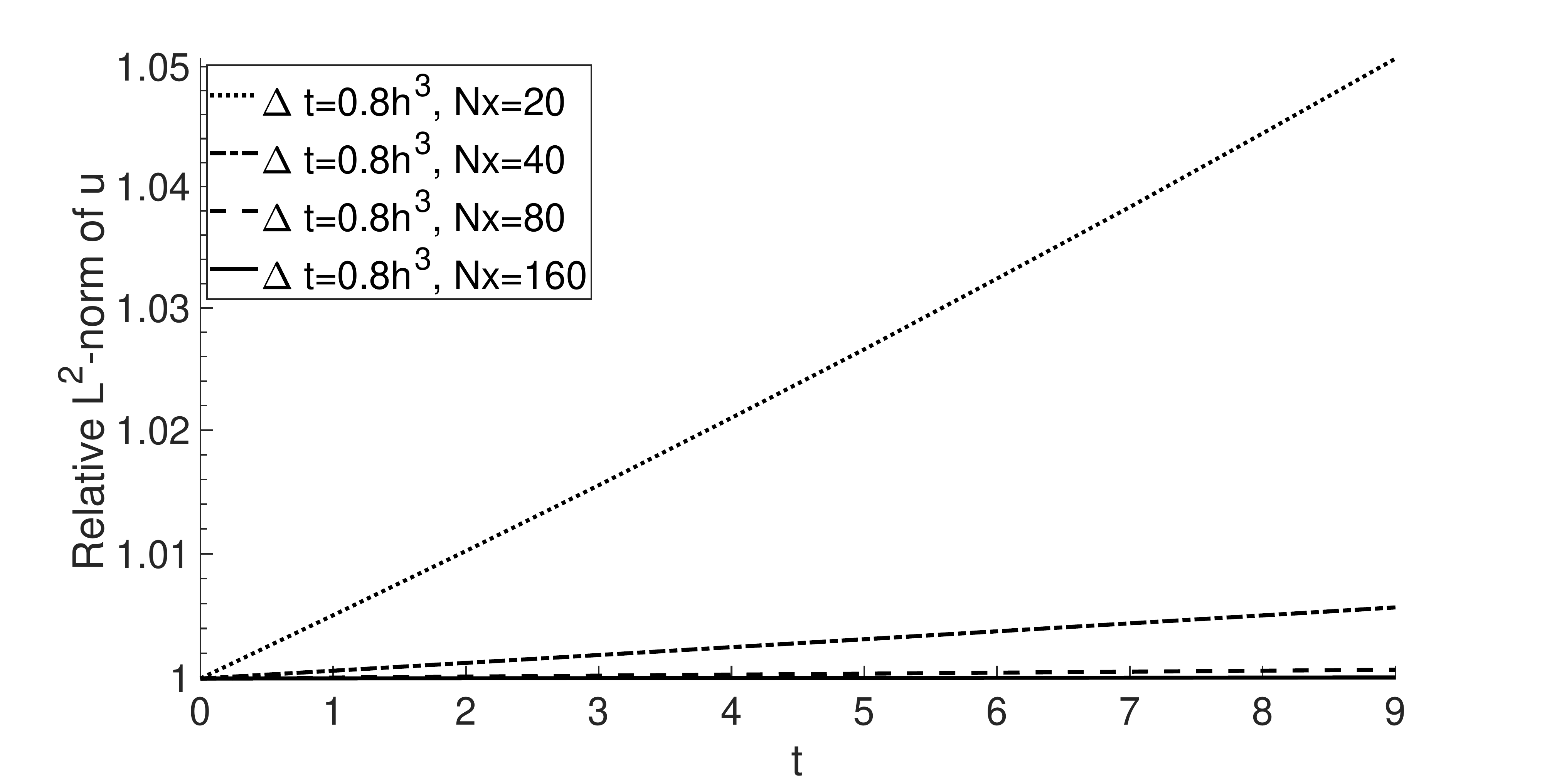}
    \caption{Histories of $\frac{\nrm{\bs{u}^n}}{\nrm{\bs{u}^0}}$ computed by $\Delta t=0.8h^3$.}
    \label{fg:num_l2_fe_r2_p3}
  \end{subfigure}
  \begin{subfigure}[b]{.48\textwidth}\centering
    \includegraphics[trim=1in 0 1in 0.2in, clip, width=\textwidth]{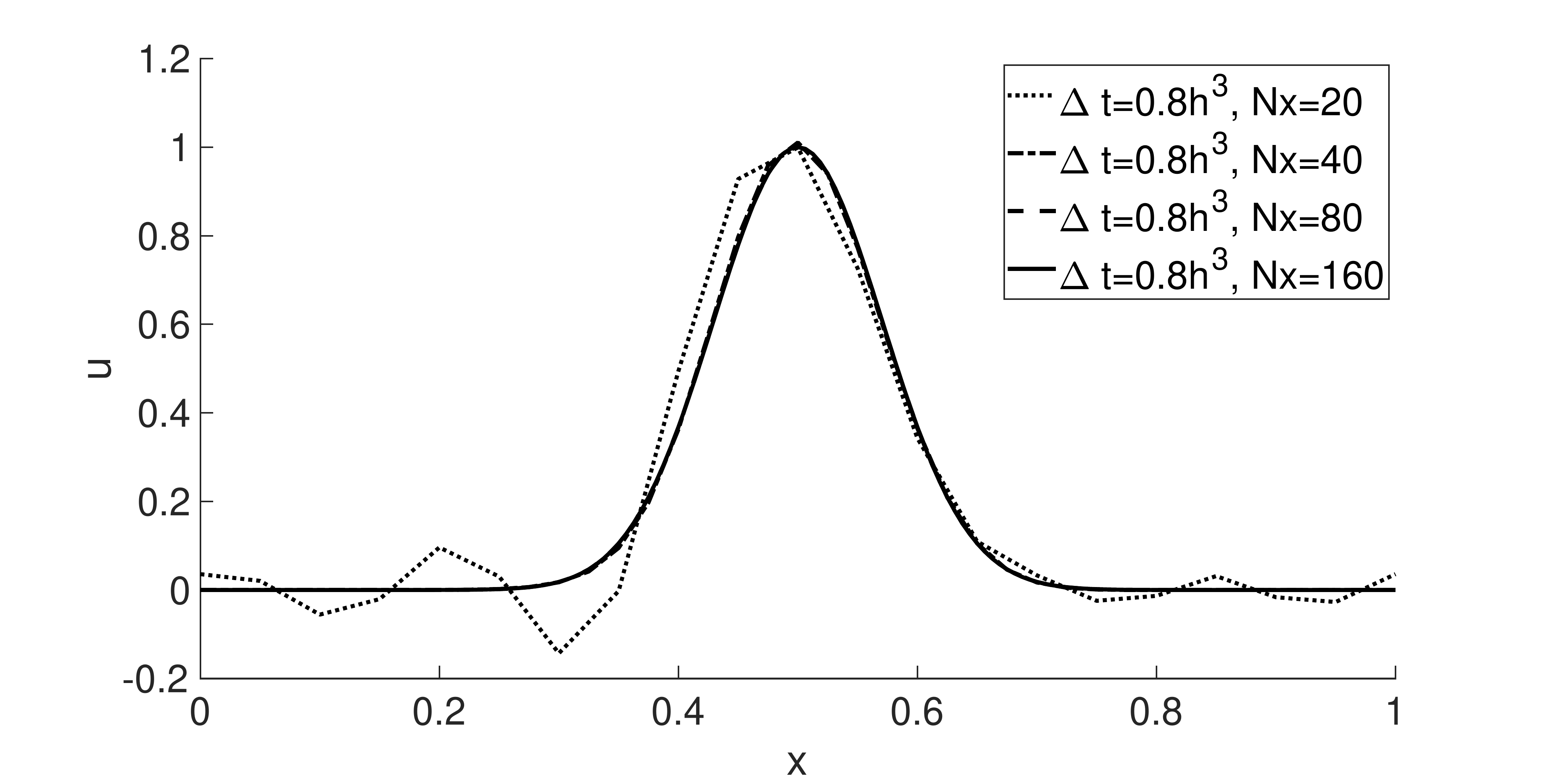}
    \caption{Nodal solutions at $T=10$ using $\Delta t=0.8h^3$.}
    \label{fg:num_l2_fe_r2_p3_plot}
  \end{subfigure}
  \vglue -.15in
  \caption{Relative $L^2$-norm of numerical solution computed by the central HV scheme with stencil $(L,R)=(2,2)$ in space and forward Euler method~\cref{eq:l2_fe} in time.
  The time step size is computed by $\Delta t=0.8h^p$, $p=1,2,3$.
  The only stable solutions are obtained when $p=3$, in which case the nodal solutions at $T=10$ are plotted in~\cref{fg:num_l2_fe_r2_p3_plot}.}
  \label{fg:num_l2_fe_r2}
\end{figure}
The corersponding plots for the HV method with $(L,R)=(7,7)$ are offered in~\cref{fg:num_l2_fe_r7}, and the same conclusions can be obtained.
\begin{figure}\centering
  \begin{subfigure}[b]{.48\textwidth}\centering
    \includegraphics[trim=1in 0 1in 0.2in, clip, width=\textwidth]{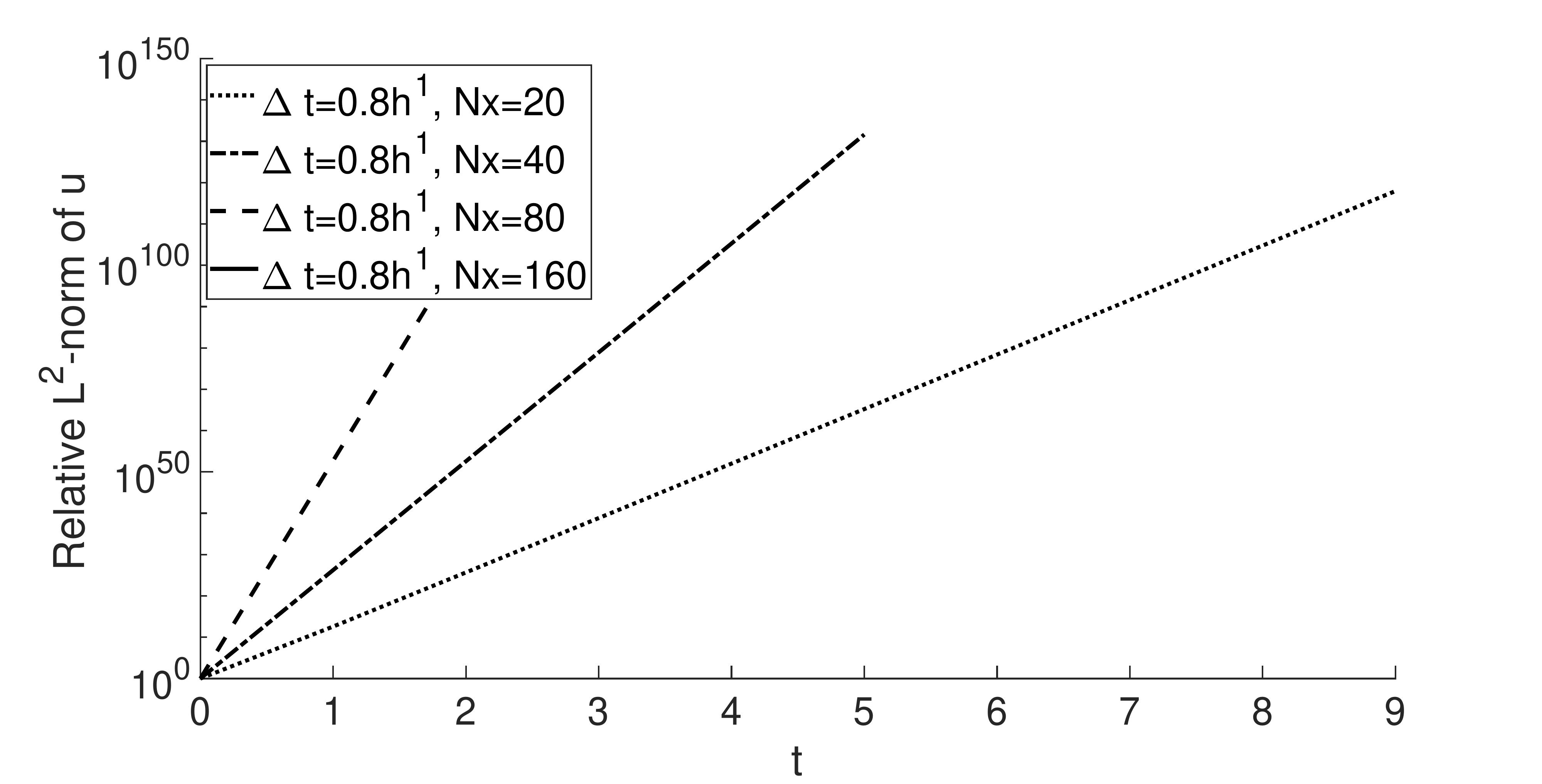}
    \caption{Histories of $\frac{\nrm{\bs{u}^n}}{\nrm{\bs{u}^0}}$ computed by $\Delta t=0.8h$.}
    \label{fg:num_l2_fe_r7_p1}
  \end{subfigure}
  \begin{subfigure}[b]{.48\textwidth}\centering
    \includegraphics[trim=1in 0 1in 0.2in, clip, width=\textwidth]{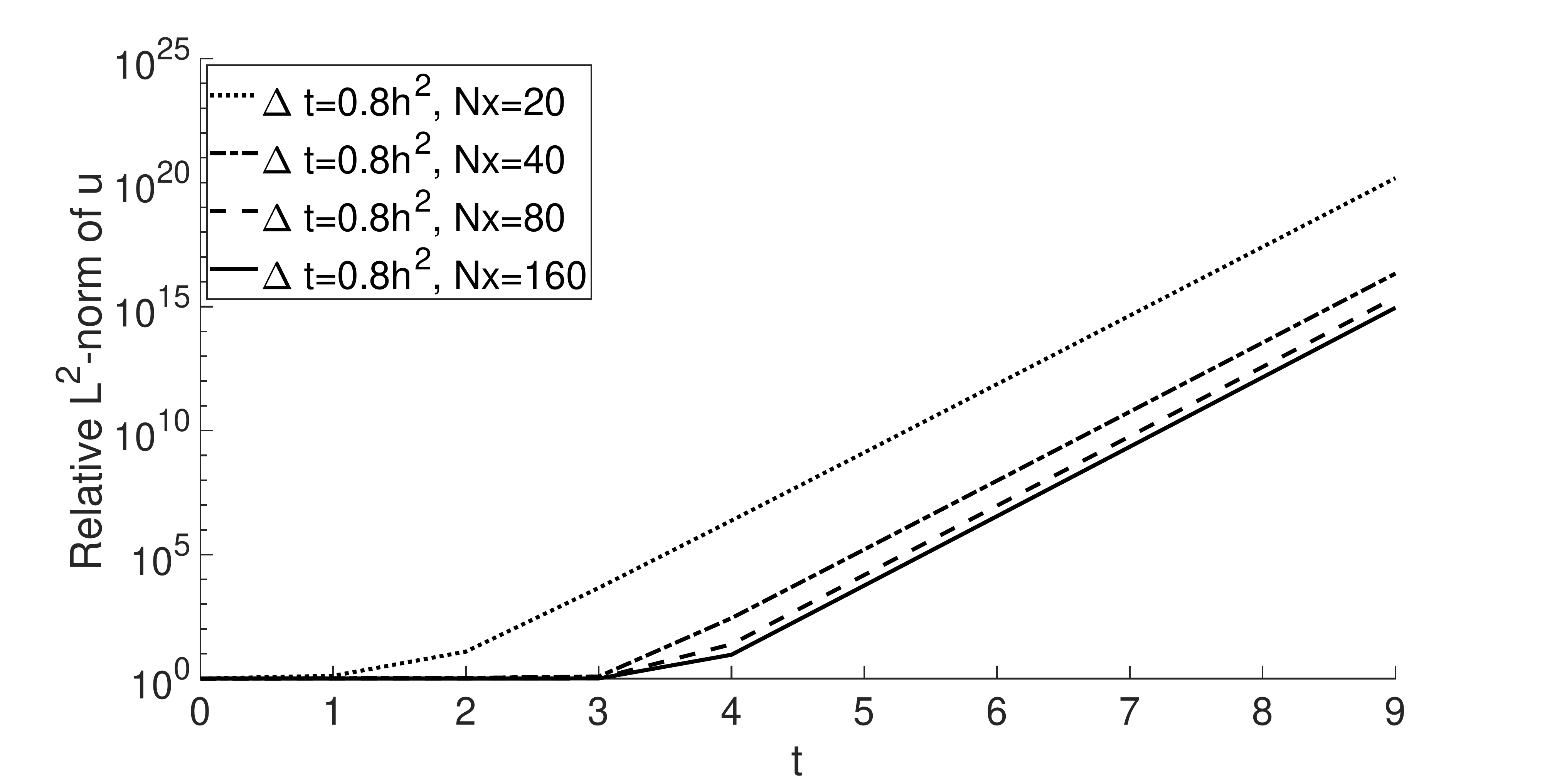}
    \caption{Histories of $\frac{\nrm{\bs{u}^n}}{\nrm{\bs{u}^0}}$ computed by $\Delta t=0.8h^2$.}
    \label{fg:num_l2_fe_r7_p2}
  \end{subfigure} \\
  \begin{subfigure}[b]{.48\textwidth}\centering
    \includegraphics[trim=1in 0 1in 0.2in, clip, width=\textwidth]{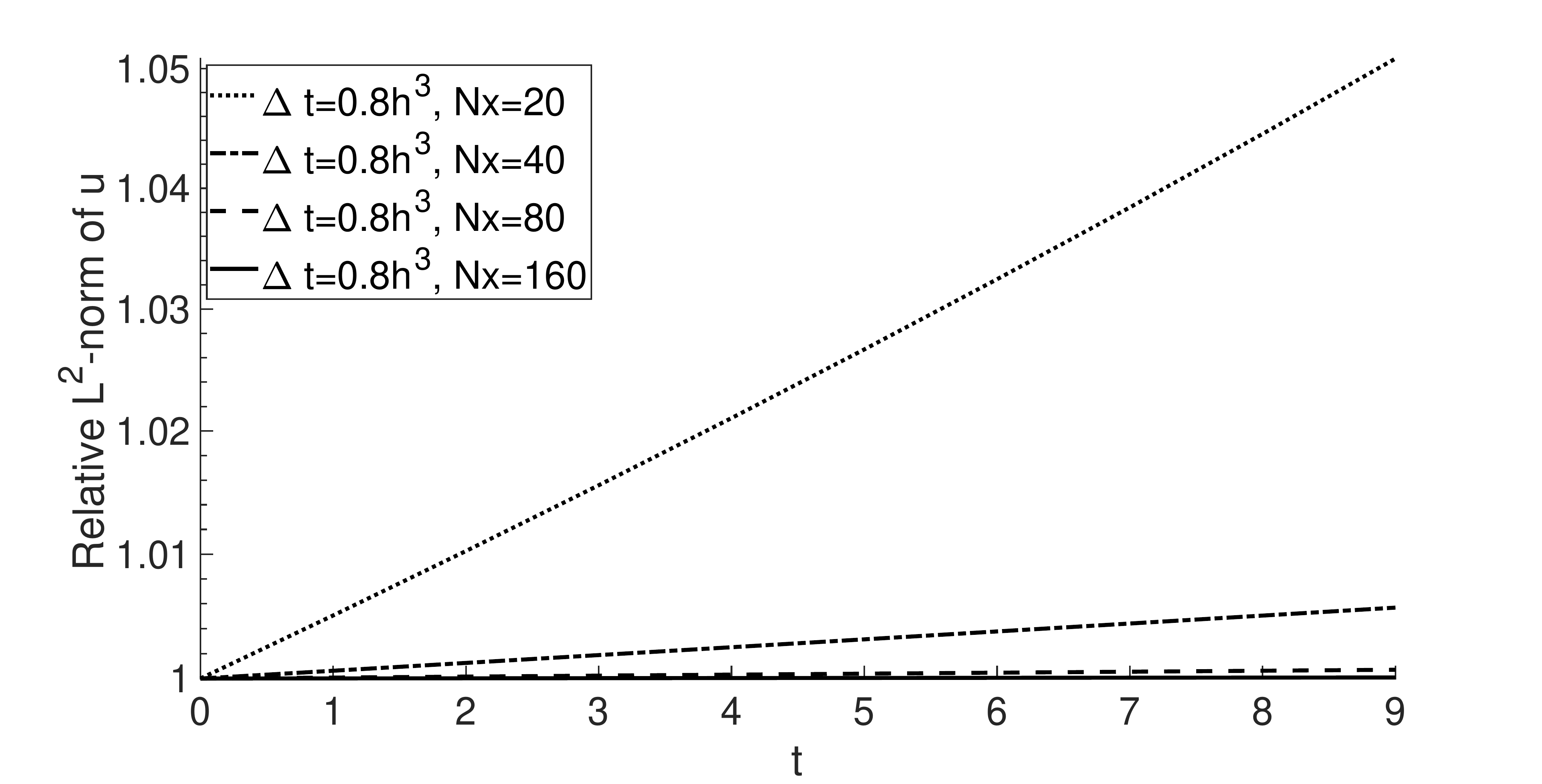}
    \caption{Histories of $\frac{\nrm{\bs{u}^n}}{\nrm{\bs{u}^0}}$ computed by $\Delta t=0.8h^3$.}
    \label{fg:num_l2_fe_r7_p3}
  \end{subfigure}
  \begin{subfigure}[b]{.48\textwidth}\centering
    \includegraphics[trim=1in 0 1in 0.2in, clip, width=\textwidth]{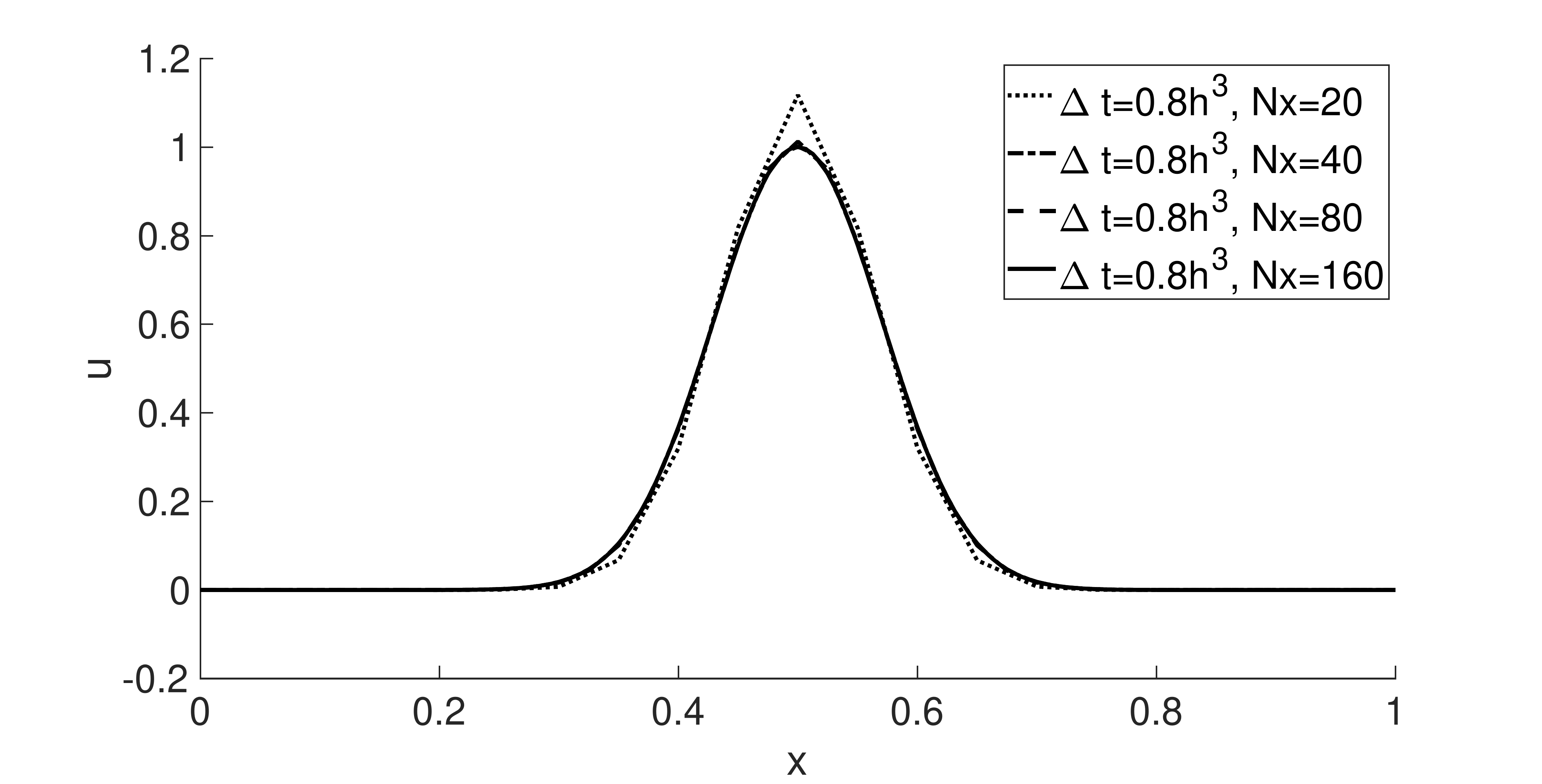}
    \caption{Nodal solutions at $T=10$ using $\Delta t=0.8h^3$.}
    \label{fg:num_l2_fe_r7_p3_plot}
  \end{subfigure}
  \vglue -.15in
  \caption{Relative $L^2$-norm of numerical solution computed by the central HV scheme with stencil $(L,R)=(7,7)$ in space and forward Euler method~\cref{eq:l2_fe} in time.
  The time step size is computed by $\Delta t=0.8h^p$, $p=1,2,3$.
  The only stable solutions are obtained when $p=3$, in which case the nodal solutions at $T=10$ are plotted in~\cref{fg:num_l2_fe_r7_p3_plot}.}
  \label{fg:num_l2_fe_r7}
\end{figure}

\section{Conclusions}
\label{sec:concl}
In this paper we fully categorize stable and unstable hybrid-variable (HV) discretizations for linear advection equations and establish the precise stability barrier of these schemes. 
The HV discretization framework belongs to the more general Hermite-type methods, and it finds approximations to both cell-averaged values and nodal values of the solution, and evolves them in time simultaneously.
The semi-discretization of cell-averaged solutions are updated in conservative form, with the flux computed using nodal solutions; whereas the semi-discretization of nodal solutions is constructed the HV approximation to spatial derivatives at nodes, called the hybrid-variable discrete differential operators (HV-DDO).
It is known that given a HV-DDO with a continuous stencil indexed by $(L,R)$, with $L$ being the number of unknowns (both cell-averaged and nodal ones) in the upwind direction and $R$ being that number in the downwind direction, the spatial order of accuracy of the resulting HV method is $L+R+1$, provided that $L>R$ and that the method is stable.

This work focuses on identifying all pairs of $(L,R)$ such that $L>R$ and the corresponding HV scheme is stable.
Particularly we proved that all methods with $L-R\ge3$ are unstable except for $(L,R)=(3,0)$, and all HV methods with $1\le L-R\le2$ are stable; hence we established the precise stability barrier of the HV discretizations.
Additionally, we proved that all central HV schemes with $L=R$ are neutrally stable in the sense that the $L^2$-norm of the semi-discretized solutions remain bounded at all times; this stability result is extended to fully-discretized HV schemes.
To this end, we fully classify all HV discretizations according to their stability at the semi-discretized level.
At last, we verify the main theoretical results by extensive numerical tests.


\bibliographystyle{plain}      
\bibliography{pap}

\appendix
\crefalias{section}{appendix}






\end{document}